\documentclass[11pt]{article}
\usepackage{amsfonts, amsmath, amssymb, amscd, amsthm, color, graphicx, mathrsfs, mathabx, wasysym, setspace, mdwlist, calc,float}
\usepackage{setspace}
\usepackage{hyperref}
\usepackage{tikz-cd} 
\allowdisplaybreaks

\usepackage{hyperref}
\usepackage{tocloft}
\usepackage{xcolor}

\usepackage{hyperref}
\hypersetup{linktocpage}

\hypersetup{colorlinks,
    linkcolor={red!50!black},
    citecolor={blue!80!black},
    urlcolor={blue!80!black}}
\usepackage{float}

\usepackage{comment}

\newcommand{\NN}{\mathbb{N}}
\newcommand{\ZZ}{\mathbb{Z}}

\newcommand{\RR}{\mathbb{R}}
\newcommand{\CC}{\mathbb{C}}

\newcommand{\inj}{\hookrightarrow}

\newcommand{\act}{\curvearrowright}

\renewcommand{\L}{\mathbf{Lab}}
\newcommand{\C}{\mathscr{C}}
\newcommand{\E}{\mathcal{E}}
\newcommand{\V}{\mathcal{V}}
\newcommand{\A}{\mathcal{A}}

\newtheorem{thm}{Theorem}[section]
\newtheorem{cor}[thm]{Corollary}
\newtheorem{lem}[thm]{Lemma}
\newtheorem{prop}[thm]{Proposition}

\newtheorem{ques}[thm]{Question}
\theoremstyle{definition}
\newtheorem{defn}[thm]{Definition}

\theoremstyle{remark}
\newtheorem{rem}[thm]{Remark}

\newcommand{\G}{\mathcal{G}}

\newcommand{\tR}{\widetilde{\mathcal{R}}}
\newcommand{\R}{\mathcal{R}}

\newcommand{\la}{\langle}
\newcommand{\ra}{\rangle}
\newcommand{\sm}{\setminus}

\newcommand{\e}{\mathfrak{e}}

\newcommand{\tPhi}{\widetilde{\Phi}}

\newcommand\blfootnote[1]{%
  \begingroup
  \renewcommand\thefootnote{}\footnote{#1}%
  \addtocounter{footnote}{-1}%
  \endgroup
}

\begin{document}

\title{Small cancellation groups are bi-exact}
\date{}
\author{Koichi Oyakawa}

\maketitle

\vspace{-10mm}

\begin{abstract}
    We prove that finitely generated (not necessarily finitely presented) $C'(\frac{1}{33})$-groups are bi-exact. This is a new class of bi-exact groups.
\end{abstract}

\section{Introduction}
\blfootnote{\textbf{MSC} Primary: 20F65. Secondary: 20F67, 20F99, 46L10.}
\blfootnote{\textbf{Key words and phrases}: bi-exact groups, small cancellation groups, proper arrays.}
Bi-exactness is an analytic property of groups defined by Ozawa in \cite{Oz04}, where it was proved that the group von Neumann algebra $L(G)$ of any non-amenable bi-exact icc group $G$ is prime (i.e. it cannot be decomposed as a tensor product of two $\rm{II}_1$ factors). Right after that, Ozawa and Popa proved a unique prime factorization theorem for non-amenable bi-exact icc groups in \cite{OP}, which was followed by substantial amount of research about consequences of bi-exactness in terms of operator algebras and ergodic theory. Non-exhaustive lists of these include measure equivalence rigidity in \cite{Sako}, amalgamated free product rigidity in \cite{CI}, and study of group actions on standard measure spaces and von Neumann algebras in \cite{HI} and \cite{HMV}. More recently, proper proximality was defined by Boutonnet, Ioana, and Peterson in \cite{BIP}, which is another analytic property of groups considered as generalization of bi-exactness, and structures of von Neumann algebras related to this class of groups were studied.

Considering the significance of bi-exactness in these fields, it is natural to ask what are examples of bi-exact groups. Below is the list of all classes of groups that are known to be bi-exact so far:
\begin{itemize}
    \item
    amenable groups,
    \item
    discrete subgroups of connected simple Lie groups of rank one (\cite[Th\'eor\`eme 4.4]{Ska}),
    \item
    hyperbolic groups (\cite[Lemma 6.2.8]{HG}),
    \item
    $\ZZ^2 \rtimes SL_2(\ZZ)$ (\cite{Oz09}).
\end{itemize}
Notably, only $\ZZ^2 \rtimes SL_2(\ZZ)$ has been newly found since Ozawa defined bi-exactness in \cite{Oz04}. While there are some results to create a bi-exact group from existing bi-exact groups by group construction (cf. \cite{Oz06b}\cite{BO}\cite{Oy}), discovering a new class of bi-exact groups is not well-studied. In this paper, we push this direction of research forward by proving the following.
\begin{thm}
\label{thm:main}
Any finitely generated (not necessarily finitely presented) $C'(\frac{1}{33})$-group is bi-exact.
\end{thm}
Theorem \ref{thm:main} provides rich source of bi-exact groups, because there are uncountably many quasi-isometry classes of finitely generated $C'(\frac{1}{33})$-groups by \cite{Bow98}. Also, note that finitely generated $C'(\frac{1}{33})$-groups are no longer hyperbolic if they are infinitely presented. Moreover, this class even contains non-relaively hyperbolic groups. The standard method to prove bi-exactness is to use topologically amenable action on some compact space that is `small at infinity'. However, this method is not applicable to prove Theorem \ref{thm:main}, because such a nice compactification of small cancellation groups has not been found unless they're hyperbolic. This is apparently why small cancellation groups were not found to be bi-exact. We overcome this difficulty by using another characterization of bi-exactness that was introduced by Chifan, Sinclair, and Udrea in \cite{CSU} and later generalized by Popa and Vaes in \cite{PV}. This characterization is based on proper arrays, which can be considered as a weaker version of quasi-cocycles, and turned out to be compatible with geometric nature of groups in \cite{Oy}. More specifically, we prove Proposition \ref{prop:intro proper array} below, which is a major step for proving Theorem \ref{thm:main}.
\begin{prop}\label{prop:intro proper array}
Let $G=\la X \mid \R \ra$ be a finitely generated $C'(\frac{1}{33})$-group such that every word $r \in \R$ satisfies $|r| > 33$ and for every letter $x\in X$, there exists $r\in \R$ such that $x$ appears in the word $r$. Let $\lambda_G \colon G\act\ell^1(G)$ be the left regular representation of $G$. Then, there exists a map $\Phi \colon G\times G\ni (g,h) \mapsto \Phi[g,h]\in\ell^1(G)$ satisfying the following five conditions.
\begin{itemize}
    \item[(1)]
    For any $g,h\in G$, $\Phi[g,h]$ has finite support and $\Phi[g,h](k)\ge 0$ for any $k\in G$.
    \item[(2)]
    $\Phi$ is symmetric, i.e. $\Phi[g,h]=\Phi[h,g]$ for any $g,h\in G$.
    \item[(3)]
    $\Phi$ is $G$-equivariant, i.e. $\Phi[kg,kh]=\lambda_G(k)\Phi[h,g]$ for any $g,h,k\in G$.
    \item[(4)]
    $d_X(g,h)\le \|\Phi[g,h]\|_1$ for any $g,h\in G$.
    \item[(5)]
    There exists $L\in \RR_{>0}$ such that $\|\Phi[g,h]-\Phi[gk,h]\|_1\le L\cdot d_X(1,k)$ for any $g,h,k\in G$.
\end{itemize}
\end{prop}

The idea of the proof of Proposition \ref{prop:intro proper array} is to consider a Cayley graph of a $C'(\frac{1}{33})$-group to be `hyperbolic relative to loops' (This terminology is just for explanation which should be understood by heart. I don't claim any rigorous definition for this.) and construct two arrays that control `peripheral part' and `hyperbolic part'. This idea is similar to coning-off the Cayley graph. The difficult part of the proof is to reconcile the conditions (4) and (5). We make it possible by inventing a machinery in which loops sitting between two vertices interact one another (see the argument from Definition \ref{def:alpha beta gamma}). The number $\frac{1}{33}$ in Theorem \ref{thm:main} comes from the choices of constants at the beginning of Section \ref{sec:main thm}.

Another important point of Theorem \ref{thm:main} is the following immediate corollary.

\begin{cor} \label{cor:main}
    Any non-virtually cyclic finitely generated $C'(\frac{1}{33})$-group is properly proximal.
\end{cor}

\begin{proof}
    Any non-virtually cyclic finitely generated $C'(\frac{1}{33})$-group $G$ is acylindrically hyperbolic by \cite[Theorem 1.3]{GS}, hence contains a free group of rank 2. This and Theorem \ref{thm:main} imply that $G$ is non-amenable and bi-exact. By \cite{BIP}, any non-amenable bi-exact group is properly proximal.
\end{proof}

This supports a positive answer to the following folklore question. There are also other results that support the positive answer to Question \ref{ques:folklore} (cf. \cite{CJJ} \cite{CS22}).

\begin{ques}\label{ques:folklore}
Is any acylindrically hyperbolic group properly proximal?
\end{ques}

Corollary \ref{cor:main} also answers one of questions posed in \cite{CJJ}, where they proved proper proximality of various groups acting on non-positively curved spaces and asked whether other classes of groups with features of non-positive curvature, including small cancellation groups, are properly proximal.

The paper is organized as follows. In Section \ref{sec:preliminary}, we discuss the necessary definitions and known results about bi-exact groups and $C'(\lambda)$-groups. In Section \ref{sec:construction of arrays}, we discuss the construction of two arrays. Section \ref{subsec:first array} deals with contours (cf. Definition \ref{def:Cayley}) of a Cayley graph which is considered as `peripheral part' and Section \ref{subsec:second array} deals with undirected edges of a Cayley graph which is considered as `hyperbolic part'. In Section \ref{sec:main thm}, we first prove Proposition \ref{prop:proper array} by combining the two arrays constructed in Section \ref{sec:construction of arrays} and then use it to prove Theorem~\ref{thm:main}. In Section \ref{sec:infinite}, we prove that some infinitely generated small cancellation groups are bi-exact in Theorem \ref{thm:infinite case} by generalizing Theorem \ref{thm:main}.

\vspace{2mm}

\noindent\textbf{Acknowledgment.}
I would like to thank Denis Osin for introducing this topic to me, for sharing his ideas, and for many helpful discussions. I would also like to thank an anonymous referee for many helpful comments.

\section{Preliminary}\label{sec:preliminary}
\subsection{Bi-exact groups}

In this section, we introduce some equivalent conditions of bi-exactness.
\begin{defn}
A group $G$ is called \emph{exact} if there exists a compact Hausdorff space on which $G$ acts topologically amenably and \emph{bi-exact} if it is exact and there exists a map $\mu\colon G\to{\rm Prob}(G)$ such that for any $s,t\in G$, we have
\[\lim_{x\to \infty}\|\mu(sxt)-s.\mu(x)\|_1=0.\]
\end{defn}

The definition of an array given below was suggested in \cite[Definition 2.1]{CSU}.

\begin{defn}\label{def:array}
Suppose that $G$ is a group, $\mathcal{K}$ a Hilbert space, and $\pi\colon G\to \mathcal{U}(\mathcal{K})$ a unitary representation. We denote $\pi(g) \in \mathcal{U}(\mathcal{K})$ by $\pi_g$ for each $g \in G$. A map $r\colon G\to \mathcal{K}$ is called an \emph{array} on $G$ into $(\mathcal{K},\pi)$, if $r$ satisfies (1) and (2) below. When there exists such $r$, we say that $G$ admits an array into $(\mathcal{K},\pi)$.
\begin{itemize}
    \item [(1)]
    $\pi_g(r(g^{-1}))=r(g)$ for all $g\in G$.
    \item [(2)]
    For every $g\in G$, we have $\sup_{h\in G}\|r(gh)-\pi_g(r(h))\|<\infty$.
\end{itemize}
If, in addition, $r$ satisfies (3) below, $r$ is called \emph{proper}.
\begin{itemize}
    \item [(3)]
    For any $N\in \NN$, $\{g\in G\mid \|r(g)\|\le N\}$ is finite.
\end{itemize}
\end{defn}

\begin{rem}
    More precisely, an array as in Definition \ref{def:array} is called a \emph{symmetric array} in \cite{CSU}. If a map $r\colon G\to \mathcal{K}$ satisfies
    \begin{itemize}
        \item [(1')]
        $\pi_g(r(g^{-1}))=-r(g)$ for all $g\in G$
    \end{itemize}
    and Definition \ref{def:array} (2), then $r$ is called an \emph{anti-symmetric array}. In fact, we will construct symmetric arrays of a $C'(\frac{1}{33})$-group in Section \ref{sec:construction of arrays} and Proposition \ref{prop:proper array}. Anti-symmetric arrays are used only in Lemma \ref{lem:free product} in this paper.
\end{rem}

Proposition \ref{prop:bi-exact} (2) and (3) are simplified versions of \cite[Proposition 2.3]{CSU} and \cite[Proposition 2.7]{PV} respectively.

\begin{prop}\label{prop:bi-exact}
Suppose that $G$ is a countable group. Then, (1), (2), and (3) below are equivalent.
\begin{itemize}
    \item[(1)]
    $G$ is bi-exact.
    \item[(2)]
    $G$ is exact and admits a proper array into the left regular representation $(\ell^2(G),\lambda_G)$.
    \item[(3)]
    $G$ is exact, and there exist an orthogonal representation $\eta\colon G \to \mathcal{O}(K_\RR)$ to a real Hilbert space $K_\RR$ that is weakly contained in the regular representation of $G$ and a map $c\colon G\to K_\RR$ that is proper and satisfies
\[\sup_{k\in G}\|c(gkh)-\eta_g c(k)\|<\infty\] 
    for all $g,h\in G$.
\end{itemize}
\end{prop}

\subsection{Small cancellation groups}
For a set of letters $X$, we denote by $F(X)$ the free group generated by $X$. Elements of $F(X)$ are reduced words in $X$. We call a subset $\R$ of $F(X) \sm \{1\}$ \emph{symmetric}, if for any $r\in \R$, $r$ is cyclically reduced and all cyclic permutations of $r$ and $r^{-1}$ are contained in $\R$. For $u,v \in \R$ such that $u \neq v$ in $F(X)$, a word $w\in F(X) \sm \{1\}$ is called the \emph{piece} for $u$ and $v$, if $w$ is the largest common prefix of $u$ and $v$ (i.e. $u$ and $v$ are decomposed into $u=wu'$ and $v=wv'$ with $u',v'\in F(X)$ and the initial letters of $u'$ and $v'$ are distinct). In the context of group presentations, elements in $\R$ are also called \emph{relations}.

\begin{defn}
\label{def:C'-condition}
For a set of letters $X$ and a real number $\lambda>0$, a symmetric set of relations $\R$ is said to satisfy \emph{$C'(\lambda)$-condition}, if for any two distinct relations $u,v \in\R$, the piece $w$ for $u$ and $v$ satisfies
\[
|w|<\lambda \min\{|u|,|v|\},
\]
where $|\cdot|$ denotes word length.
\end{defn}

\begin{defn} \label{def:C'-group}
For a real number $\lambda>0$, a group $G$ is called a \emph{$C'(\lambda)$-group}, if $G$ has a group presentation $G=\la X \mid \R \ra$ such that $\R$ is symmetric and satisfies $C'(\lambda)$-condition. If in addition, $X$ is finite, then $G$ is called a \emph{finitely generated $C'(\lambda)$-group}.
\end{defn}

Note that $\R$ can be infinite. In the following, we state known facts of $C'(\lambda)$-groups. Lemma \ref{lem:Green} is a simplified version of \cite[Theorem 4.4]{LS} which is also known as Greendlinger's lemma.

%This also follows from p.241 of \cite{Gdela}.

\begin{lem}[Greendlinger's lemma]\label{lem:Green}
Suppose that $G=\la X \mid \R \ra$ is a $C'(\lambda)$-group with $\lambda \le \frac{1}{6}$, then any word $w \in F(X) \sm \{1\}$ satisfying $w=_G 1$ contains a subword $s$ of some relation $r \in \R$ such that
\[
|s|>(1-3\lambda)|r|.
\]
\end{lem}

\begin{rem}
Note that in Lemma \ref{lem:Green}, $w$ is reduced but doesn't need to be cyclically reduced, and $w$ itself, not some cyclic permutation of $w$, contains $s$.
\end{rem}

Before stating more results, we prepare terminologies and notations. Our definition of graphs and Cayley graphs follows Serre. Therefore, we allow loops and multiple edges by default. A positive edge set of a graph means a set of edges where an orientation is assigned to each edge. Readers are referred to \cite[Section 2.1]{Ser03} for details.

\begin{defn}\label{def:Cayley}
Suppose $G=\la X \mid \R \ra$ is a $C'(\lambda)$-group with $\lambda > 0$. The \emph{Cayley graph} $\Gamma$ of $G$ (with respect to $X$) is a graph whose vertex set $V(\Gamma)$ is $G$ and positive edge set $E^+(\Gamma)$ is $G \times X$. Each positive edge $(g,s) \in G \times X$ goes from $g$ to $gs$ and has a label $\L((g,s))=s$. The inverse edge $\overline{(g,s)}$ of $(g,s)$ has a label $\L(\overline{(g,s)})=s^{-1}$. We denote by $\E$ the set of undirected edges of $\Gamma$. (Note $\E$ has a bijective correspondence with $E^+(\Gamma)$.) A \emph{path} $p$ in $\Gamma$ is a graph homomorphism (ignoring directions and labels of edges of $\Gamma$) from $[0,n]$ to $\Gamma$ where $n \in \NN \cup \{0\}$ and $[0,n]$ is considered as a graph with $n$ edges.  By abuse of notation, we also denote by $p$ the image of a path $p$, which is a subgraph of $\Gamma$. A \emph{loop} is a path $p$ such that $p(0)=p(n)$. For a path $p$, its label $\L(p)$ is defined by $\L(p)=\L(p[0,1])\L(p[1,2])\cdots\L(p[n-1,n])$ and we denote 
\[
p_- = p(0) {\rm ~~~and~~~} p_+ = p(n).
\]
A path $p$ is called \emph{simple}, if we have $p(k)\neq p(\ell)$ unless $|k-\ell|=0,n$. A path is called an \emph{arc}, if it's injective. If $p$ is an arc and $g,h\in V(\Gamma)$ are two vertices on $p$ satisfying $g=p(k)$, $h=p(\ell)$, and $k \le \ell$, then we denote the subpath $p|_{[k,l]}$ by $p_{[g,h]}$. A \emph{contour} is an image of a loop $r$ such that $\L(r)\in \R$. Because $\R$ is symmetric, the definition of contours doesn't depend on the initial point and direction of $r$. That is, if a loop $r \colon [0,n] \to \Gamma$ satisfies $\L(r) \in \R$, then the loops $p_k, p^{-1} \colon [0,n] \to \Gamma$, where $k \in \NN$, defined by $p_k(i) = p(i+k~\text{mod}~n)$ and $p^{-1}(i) = p(n-i)$ also satisfy $\L(p_k), \L(p^{-1}) \in \R$. We denote the set of all contours by $\C$.
\end{defn}

\begin{rem}
\label{rem:without inversion}
By our definition of the Cayley graph, $G$ acts on $\Gamma$ without inversion of edges.
\end{rem}

\begin{rem}
\label{rem:simple}
When $\lambda \le \frac{1}{6}$, any contour is a simple loop (i.e. it has no self-intersection) by Greendlinger's lemma. The set $\C$ of contours can be considered as a subset of the power set $2^{\E}$ of $\E$. Also, $G$ acts naturally on $\C$.
\end{rem}

\begin{defn}\label{def:E,V}
For any subgraph $Y$ of $\Gamma$, we define $\E(Y)$ to be the set of all undirected edges that are contained in $Y$, and $\V(Y)$ to be the set of all vertices of $Y$.
\end{defn}

\begin{rem}\label{rem:E}
Note $\E=\E(\Gamma)$ in Definition \ref{def:Cayley}. By abuse of notation and to simplify notations, for a subgraph $Y$ of $\Gamma$ (e.g. paths, contours, or their intersections), we often denote the cardinality of $\E(Y)$ by $|Y|$ (i.e. $|Y|=|\E(Y)|$).
\end{rem}

The following is an immediate corollary of Greendlinger's lemma.

\begin{cor}
\label{cor:Green}
Suppose that $G=\la X \mid \R \ra$ is a $C'(\lambda)$-group with $\lambda \le \frac{1}{6}$ and $\Gamma$ is its Cayley graph. If $p$ is a loop in $\Gamma$ such that $\L(p) \neq_{F(X)} 1$, then there exists a contour $r$ such that
\[
|p \cap r| > (1-3\lambda)|r|.
\]
If, in addition, $\L(p)$ is reduced, then there exists a subpath $s$ of $p$ such that $s \subset p \cap r$ and $|s| > (1-3\lambda)|r|$.
\end{cor}

\begin{proof}
The reduced word of $\L(p)$ is not 1 in $F(X)$ since $\L(p) \neq_{F(X)} 1$, hence we can apply Lemma \ref{lem:Green} to the reduced word. If $\L(p)$ is already reduced, then apply Lemma \ref{lem:Green} directly to $\L(p)$.
\end{proof}

\begin{rem}
We will often consider loops as a map from a simple loop to $\Gamma$ in order not to fix a starting point of a path.
\end{rem}

Lemma \ref{lem:AD} (1), (2), (3) is \cite[Lemma 4.2]{AD} and (4) is an immediate corollary of (1) and (2).
\begin{lem}
\label{lem:AD}
Suppose that $G=\la X \mid \R \ra$ is a $C'(\lambda)$-group with $\lambda \le \frac{1}{8}$ and $\Gamma$ is its Cayley graph. Let $r$ be a contour and $v,w$ be two vertices on $r$.
\begin{itemize}
    \item[(1)]
    If $s$ is a subpath of $r$ connecting $v$ and $w$ and satisfies $|s|<\frac{1}{2}|r|$, then $s$ is a unique geodesic path in $\Gamma$ from $v$ to $w$.
    
    \item[(2)]
    If $s_1$ and $s_2$ are two distinct subpaths of $r$ connecting $v$ and $w$ and satisfy $|s_1|=|s_2|=\frac{1}{2}|r|$, then $s_1$ and $s_2$ are the only geodesic paths in $\Gamma$ from $v$ to $w$.
    
    \item[(3)]
    For any contour $r$ and any geodesic path $p$, $r\cap p$ is an image of one arc, if it's nonempty.
    
    \item[(4)]
    For any two distinct contours $r_1$ and $r_2$, $r_1 \cap r_2$ is an image of one arc, if it's nonempty.
\end{itemize}
\end{lem}

\begin{proof}[Proof of Lemma \ref{lem:AD} (4)]
    Let $(x,y)$ be a pair of vertices in $r_1 \cap r_2$ such that $d(x,y)$ is the maximum among all the pairs of vertices in $r_1 \cap r_2$. Let $s_1$ and $t_1$ (resp. $s_2$ and $t_2$) be the two distinct subpaths of $r_1$ (resp. $r_2$) connecting $x$ and $y$. Suppose $d(x,y) = \frac{1}{2}|r_1|$ for contradiction, then $s_1$ and $t_1$ become the only geodesic paths in $\Gamma$ from $x$ to $y$ by applying Lemma \ref{lem:AD} (2) to $r_1$. On the other hand, one of $s_2$ or $t_2$ is a geodesic path in $\Gamma$ from $x$ to $y$ by applying Lemma \ref{lem:AD} (1), (2) to $r_2$. Hence, $\{s_1,t_1\}\cap\{s_2,t_2\} \neq \emptyset$. This contradicts $C'(\lambda)$-condition by the assumption $d(x,y) = |s_1| = |t_1| = \frac{1}{2}|r_1|$. Hence, $d(x,y) < \frac{1}{2}|r_1|$. We assume $|s_1| < \frac{1}{2}|r_1|$ without loss of generality, then $s_1$ is a unique geodesic path in $\Gamma$ from $x$ to $y$ by Lemma \ref{lem:AD} (1). Arguing similarly for $r_2$, we may assume that $s_2$ is a unique geodesic path in $\Gamma$ from $x$ to $y$. Hence, $s_1 = s_2$. By $C'(\lambda)$-condition, we have $|s_1| = |s_2| < \lambda \min\{|r_1|,|r_2|\}$.

    We claim  $s_1 = s_2 = r_1 \cap r_2$. Indeed, suppose for contradiction that there exists a vertex $z$ in $t_1 \cap r_2$ with $z \notin \{x,y\}$. Applying the same argument as $(x,y)$ above to the pair $(x,z)$, there exists a subpath $u$ of $r_1$ that is a unique geodesic path from $z$ to $x$ and satisfies $|u| < \lambda |r_1|$. If $u \cap s_1 = \{x\}$, then we have $d(z,y) = |u|+|s_1| > d(x,y)$ by $|u|+|s_1| < 2\lambda|r_1| < \frac{1}{2}|r_1|$ and Lemma \ref{lem:AD} (1). This contradicts maximality of $d(x,y)$. If $s_1 \subset u$, then we have $d(z,y) = |u| > |s_1| = d(x,y)$, which again contradicts maximality of $d(x,y)$. Thus, $r_1 \cap r_2$ is the arc $s_1 (=s_2)$.
\end{proof}

\begin{rem}
\label{rem:AD}
The definition of $C'(\lambda)$-condition in \cite[Definition 3.2]{AD} puts an additional condition other than our Definition \ref{def:C'-condition}, but the proof of Lemma \ref{lem:AD} doesn't use the additional condition. Indeed, the proof is based on Greendinger's lemma.
\end{rem}

\begin{rem}\label{rem:r cap p}
For a contour $r$ and a geodesic $p$, $r\cap p$ is a subpath of $p$, if it's nonempty, by Lemma \ref{lem:AD} (3). Hence, by assigning to $r\cap p$ the same direction as $p$, we denote the two endpoints of $r\cap p$ by $(r\cap p)_-$ and $(r\cap p)_+$. Note they satisfy $d(p_-,(r\cap p)_-)\le d(p_-,(r\cap p)_+)$.
\end{rem}

Theorem \ref{thm:classification} is an immediate consequence of \cite[Theorem 43]{Gdela} when $\lambda \le \frac{1}{8}$.

\begin{figure}
  \begin{center}
 \hspace{0mm} %% Creator: Inkscape 1.2 (dc2aedaf03, 2022-05-15), www.inkscape.org
%% Accompanies image file 'B.pdf' (pdf, eps, ps)
\begingroup%
  \makeatletter%
  \providecommand\color[2][]{%
    \errmessage{(Inkscape) Color is used for the text in Inkscape, but the package 'color.sty' is not loaded}%
    \renewcommand\color[2][]{}%
  }%
  \providecommand\transparent[1]{%
    \errmessage{(Inkscape) Transparency is used (non-zero) for the text in Inkscape, but the package 'transparent.sty' is not loaded}%
    \renewcommand\transparent[1]{}%
  }%
  \providecommand\rotatebox[2]{#2}%
  \newcommand*\fsize{\dimexpr\f@size pt\relax}%
  \newcommand*\lineheight[1]{\fontsize{\fsize}{#1\fsize}\selectfont}%
  \ifx\svgwidth\undefined%
    \setlength{\unitlength}{251.12479921bp}%
    \ifx\svgscale\undefined%
      \relax%
    \else%
      \setlength{\unitlength}{\unitlength * \real{\svgscale}}%
    \fi%
  \else%
    \setlength{\unitlength}{\svgwidth}%
  \fi%
  \global\let\svgwidth\undefined%
  \global\let\svgscale\undefined%
  \makeatother%
  \begin{picture}(1,0.19513649)%
    \lineheight{1}%
    \setlength\tabcolsep{0pt}%
    \put(0,0){\includegraphics[width=\unitlength,page=1]{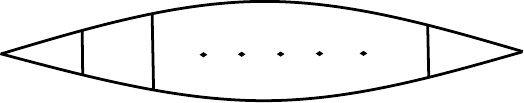}}%
    \put(0.06749901,0.17275164){\color[rgb]{0,0,0}\makebox(0,0)[lt]{\lineheight{1.25}\smash{\begin{tabular}[t]{l}$B$\end{tabular}}}}%
  \end{picture}%
\endgroup%

  \end{center}
   \vspace{-3mm}
  \caption{Simple geodesic bigon}
  \label{FigB}
\end{figure}

\begin{figure}
  \begin{center}
 \hspace{18mm} 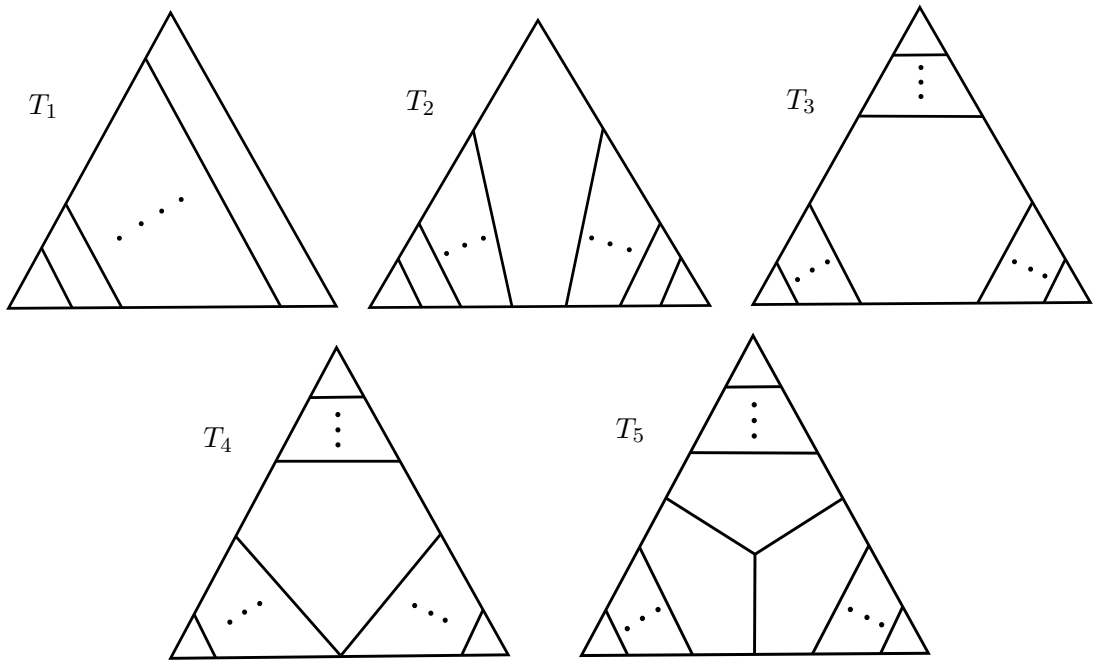
  \end{center}
  \vspace{-3mm}
  \caption{Simple geodesic triangles}
  \label{Fig2}
\end{figure}

\begin{thm}
\label{thm:classification}
Suppose that $G=\la X \mid \R \ra$ is a $C'(\lambda)$-group with $\lambda \le \frac{1}{8}$ and $\Gamma$ is its Cayley graph.
\begin{itemize}
    \item[(1)]
    If $\Delta$ is a reduced van Kampen diagram of a simple geodesic bigon with two distinct vertices in $\Gamma$ that has more than one 2-cell, then $\Delta$ has the form $B$ in Figure \ref{FigB}.
    
    \item[(2)]
    If $\Delta$ is a reduced van Kampen diagram of a simple geodesic triangle with three distinct vertices in $\Gamma$ that has more than one 2-cell, then $\Delta$ has one of the forms $T_1, T_2, T_3, T_4, T_5$ in Figure \ref{Fig2}.
\end{itemize}
\end{thm}

\begin{rem}
The form $T_2$ in Figure \ref{Fig2} can occur. Indeed, take a simple geodesic bigon and add a third vertex in the middle of one geodesic path. This turns the bigon into a simple geodesic triangle of form $T_2$.
\end{rem}

\section{Construction of symmetric arrays}\label{sec:construction of arrays}
From Section \ref{sec:construction of arrays} up to the end of Remark \ref{rem:proper array}, suppose that $G=\la X \mid \R \ra$ is a finitely generated $C'(\lambda)$-group satisfying the following condition \eqref{eq:*} :
\begin{equation} \tag{$\ast$}\label{eq:*}
{\rm For ~any ~} r \in \R ,~|r| > \frac{1}{\lambda} {\rm ~and ~for ~any ~letter ~} x\in X, {\rm ~there ~exists ~} r\in \R {\rm ~such ~that ~}x\in r.
\end{equation}
Here, we define $x \in r \iff \text{the letter $x$ appears in the word $r$}$. The condition \eqref{eq:*} corresponds to that of Proposition \ref{prop:intro proper array}. We also suppose that $\Gamma$ is the Cayley graph of $G$ with respect to $X$ and fix two real numbers $\lambda,\mu$ by
\[
\lambda=\frac{1}{33} {\rm ~~~and~~~} \mu=\frac{4}{33}=4\lambda.
\]

We endow $\Gamma$ with its graph metric $d$.

\begin{rem}
\label{rem:relation}
The condition \eqref{eq:*} is not restrictive in proving Theorem \ref{thm:main}. Indeed, for any finitely generated $C'(\lambda)$-group $G=\la X \mid \R \ra$, define
\[
\R_1=\Big\{r\in \R \;\Big|\; |r| > \frac{1}{\lambda}\Big\}~~{\rm and}~~X_1=\{x\in X\mid x~{\rm appears~in~some}~r \in\R_1 \}.
\]
We also denote $\R_1^c=\R \sm \R_1$ and $X_1^c=X \sm X_1$. $\la X_1 \mid \R_1 \ra$ is a finitely generated $C'(\lambda)$-group satisfying \eqref{eq:*}. Note that for any $r \in \R_1^c$, any letter $x$ appearing in $r$ cannot be in $X_1$ by $C'(\lambda)$-condition and $|r|\le \frac{1}{\lambda}$, hence we have $x\in X_1^c$. Thus, the presentation $\la X_1^c \mid \R_1^c \ra$ is well-defined and it's a finitely presented $C'(\lambda)$-group, which is hyperbolic, hence bi-exact. Since we have
\[
G=\la X \mid \R \ra = \la X_1 \mid \R_1 \ra * \la X_1^c \mid \R_1^c \ra,
\]
$G$ is bi-exact if and only if $\la X_1 \mid \R_1 \ra$ is bi-exact by \cite[Theorem 1.1]{Oy}.
\end{rem}

Lemma \ref{lem:Fig3new} and \ref{lem:Fig4new} are immediate corollaries of Theorem \ref{thm:classification}.

\begin{lem}
\label{lem:Fig3new}
Suppose $g,h\in G$ and $p,q$ are geodesic paths in $\Gamma$ from $g$ to $h$. Then, the word $\L(qp^{-1})$ has a van Kampen diagram $\Delta$ of the form in Figure \ref{Fig3new}.
\end{lem}

\begin{proof}
Let $g=v_0,\cdots,v_n=h$ be vertices of $p\cap q$ such that $d(g,v_0)\le \cdots \le d(g,v_n)$. Then, for each $i\in \{1,\cdots,n\}$, either $p_{[v_{i-1},v_i]}=q_{[v_{i-1},v_i]}$ is satisfied or $q_{[v_{i-1},v_i]}(p_{[v_{i-1},v_i]})^{-1}$ is a simple geodesic bigon.
\end{proof}

\begin{lem}
\label{lem:Fig4new}
Suppose $g,h\in G$, $x\in X\cup X^{-1}$, $p$ is a geodesic path in $\Gamma$ form $g$ to $h$, and $q$ is a geodesic path in $\Gamma$ from $gx$ to $h$, $\e$ is the edge of $\Gamma$ from $g$ to $gx$ whose label is $x$. Then, the word $\L(\e qp^{-1})$ has a van Kampen diagram $\Delta$ of one of the forms in Figure \ref{Fig4new}.
\end{lem}

\begin{proof}
If $gx\notin p$ and $g\notin q$, let $v\in G$ be the furthest vertex from $h$ in $p\cap q$, then $\e q_{[gx,v]}(p_{[g,v]})^{-1}$ is a simple geodesic triangle of the form $T_1$ in Figure \ref{Fig2}. Hence, $\Delta$ becomes form (a) with Lemma \ref{lem:Fig3new} applied to $q_{[v,h]}(p_{[v,h]})^{-1}$. If $gx\in p$, $\Delta$ becomes form (b) with Lemma \ref{lem:Fig3new} applied to $q(p_{[gx,h]})^{-1}$. If $g\in q$, $\Delta$ becomes form (c) with Lemma \ref{lem:Fig3new} applied to $q_{[g,h]}p^{-1}$.
\end{proof}

\begin{figure}
  \begin{center}
 \hspace{0mm} %% Creator: Inkscape 1.2 (dc2aedaf03, 2022-05-15), www.inkscape.org
%% Accompanies image file 'Fig3new.pdf' (pdf, eps, ps)
\begingroup%
  \makeatletter%
  \providecommand\color[2][]{%
    \errmessage{(Inkscape) Color is used for the text in Inkscape, but the package 'color.sty' is not loaded}%
    \renewcommand\color[2][]{}%
  }%
  \providecommand\transparent[1]{%
    \errmessage{(Inkscape) Transparency is used (non-zero) for the text in Inkscape, but the package 'transparent.sty' is not loaded}%
    \renewcommand\transparent[1]{}%
  }%
  \providecommand\rotatebox[2]{#2}%
  \newcommand*\fsize{\dimexpr\f@size pt\relax}%
  \newcommand*\lineheight[1]{\fontsize{\fsize}{#1\fsize}\selectfont}%
  \ifx\svgwidth\undefined%
    \setlength{\unitlength}{371.4395941bp}%
    \ifx\svgscale\undefined%
      \relax%
    \else%
      \setlength{\unitlength}{\unitlength * \real{\svgscale}}%
    \fi%
  \else%
    \setlength{\unitlength}{\svgwidth}%
  \fi%
  \global\let\svgwidth\undefined%
  \global\let\svgscale\undefined%
  \makeatother%
  \begin{picture}(1,0.12756895)%
    \lineheight{1}%
    \setlength\tabcolsep{0pt}%
    \put(0,0){\includegraphics[width=\unitlength,page=1]{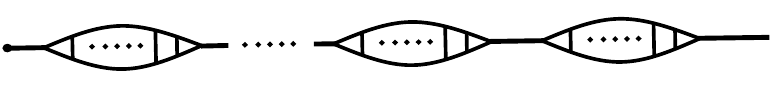}}%
    \put(0.9810558,0.10467597){\color[rgb]{0,0,0}\rotatebox{-2.75874555}{\makebox(0,0)[lt]{\lineheight{1.25}\smash{\begin{tabular}[t]{l}$h$\end{tabular}}}}}%
    \put(-0.0006702,0.09846191){\color[rgb]{0,0,0}\rotatebox{0.93355431}{\makebox(0,0)[lt]{\lineheight{1.25}\smash{\begin{tabular}[t]{l}$g$\end{tabular}}}}}%
    \put(0.46728753,0.00151305){\color[rgb]{0,0,0}\rotatebox{0.57586904}{\makebox(0,0)[lt]{\lineheight{1.25}\smash{\begin{tabular}[t]{l}$q$\end{tabular}}}}}%
    \put(0.46438983,0.1215185){\color[rgb]{0,0,0}\rotatebox{-0.47516799}{\makebox(0,0)[lt]{\lineheight{1.25}\smash{\begin{tabular}[t]{l}$p$\end{tabular}}}}}%
    \put(0,0){\includegraphics[width=\unitlength,page=2]{Fig3new.pdf}}%
  \end{picture}%
\endgroup%

  \end{center}
  \vspace{-3mm}
  \caption{Geodesic bigon}
  \label{Fig3new}
\end{figure}

\begin{figure}
  \begin{center}
 \hspace{18mm} 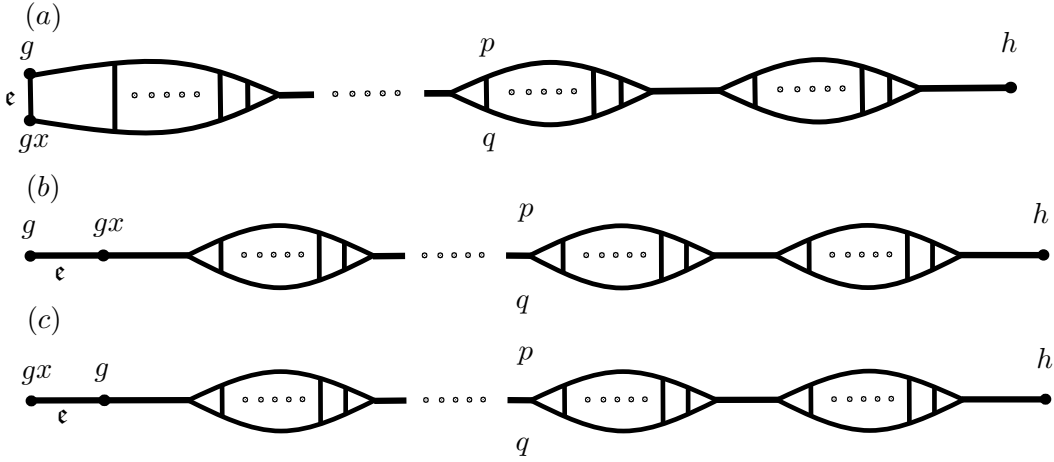
  \end{center}
  \vspace{-3mm}
  \caption{Geodesic triangles with one edge side}
  \label{Fig4new}
\end{figure}

\begin{rem}
\label{rem:CDelta}
In Lemma \ref{lem:Fig3new} and \ref{lem:Fig4new}, there exists a unique label preserving graph homomorphism from the van Kampen diagram $\Delta$ to the Cayley graph $\Gamma$ that sends the base point of $\Delta$ to $g$. We denote by $\C_\Delta$ the set of all contours that are images of some 2-cell of $\Delta$ by this graph homomorphism. Recall that $\C$ is the set of all contours.
\end{rem}

\begin{defn}
Let $p>0$. For $r \in \C$, we define $1_r \in \ell^p(\C)$ by $1_{r}(r')=1$ if $r'=r$ and $1_{r}(r')=0$ if $r'\in \C\sm\{r\}$. For $g\in G$ and $\xi=\sum_{r\in\C}c_r1_r\in \ell^p(\C)$ with $c_r\in\CC$ ($\forall\, r \in \C$), we define $\pi_\C(g)\xi \in \ell^p(\C)$ by
\[
\pi_\C(g)\xi = \sum_{r\in\C}c_r1_{gr}.
\]
This defines an isometric action $\pi_\C\colon G \act \ell^p(\C)$. Similarly, for $e \in \E$, we define $1_e \in \ell^p(\C)$ by $1_{e}(e')=1$ if $e'=e$ and $1_{e}(e')=0$ if $e'\in \E\sm\{e\}$. For $g\in G$ and $\eta=\sum_{e\in\E}c_e1_e\in \ell^p(\E)$ with $c_e\in\CC$ ($\forall\, e \in \E$), we define $\pi_\E(g)\eta \in \ell^p(\E)$ by
\[
\pi_\E(g)\eta=\sum_{e\in\E}c_e1_{ge}.
\]
This defines an isometric action $\pi_\E\colon G \act \ell^p(\E)$.
\end{defn}

\begin{rem}
We will deal with only $p=1,2$ in this paper.
\end{rem}

\subsection{First array}\label{subsec:first array}
The main goal of Section \ref{subsec:first array} is to define a function on $\C$ in Definition \ref{def:xi[g,h,psi]} and prove Proposition \ref{prop:xi array}. First of all, we assign a set of contours to a geodesic and to a pair of vertices of $\Gamma$, and define order on the set of contours. In doing so, we also study their basic properties which will be used everywhere after. This argument is up to Remark \ref{rem:order}.

\begin{defn}
For $g,h\in G$ we denote by $\G_{g,h}$ the set of all geodesic path in $\Gamma$ from $g$ to $h$.
\end{defn}

In Definition \ref{def:C_p,mu} below, $r \cap p$ is an arc by Theorem \ref{lem:AD} (4). Hence, the length $|r\cap p|$ is well-defined.

\begin{defn}\label{def:C_p,mu}
For $g,h\in G$, $p\in\G_{g,h}$, and $\mu'>2\lambda$, we define $\C_{p,\mu'}$ by
\[
\C_{p,\mu'}=\{ r\in\C \mid |r\cap p|\ge \mu'|r|\}.
\]
\end{defn}

For simplicity, we denote for $g,h\in G$ and $\mu'>0$,
\begin{equation}
\label{eq:all intersection}
    \C_{g,h,\mu'}=\bigcap_{p\in\G_{g,h}}\C_{p,\mu'}.
\end{equation}

We begin with simple observations.

\begin{lem}
\label{lem:two contours}
Suppose $g,h\in G$, $p\in\G_{g,h}$, and $r,s\in \C_{p,\lambda}$ with $r\neq s$. Then, neither of $r\cap p\subset s\cap p$ nor $s\cap p\subset r\cap p$ occurs.
\end{lem}

\begin{proof}
Suppose $r\cap p\subset s\cap p$ for contradiction, then $r\cap p \subset r\cap s$. This implies $\lambda|r| \le |r\cap p| \le |r\cap s|$ since $r\in \C_{p,\lambda}$. This contradicts $r\neq s$ by $C'(\lambda)$-condition. Thus, $r\cap p \nsubset s\cap p$. Similarly, $s\cap p \nsubset r\cap p$ follows from $s\in \C_{p,\lambda}$.
\end{proof}

\begin{rem}
\label{rem:two contours}
By Lemma \ref{lem:two contours}, for any $g,h\in G$, $p\in\G_{g,h}$, and $r,s\in \C_{p,\lambda}$ with $r\neq s$, one of the two cases below holds (cf. Remark \ref{rem:r cap p}).
\begin{itemize}
    \item
    $d(g,(r\cap p)_-)< d(g,(s\cap p)_-)$ and $d(g,(r\cap p)_+)< d(g,(s\cap p)_+)$.
    \item
    $d(g,(s\cap p)_-)< d(g,(r\cap p)_-)$ and $d(g,(s\cap p)_+)< d(g,(r\cap p)_+)$.
\end{itemize}
\end{rem}

\begin{cor}
For any $g,h\in G$ and $p\in\G_{g,h}$, $\C_{p,\lambda}$ is finite.
\end{cor}

\begin{proof}
The map $\C_{p,\lambda}\ni r \mapsto d(g,(r\cap p)_-) \in \{0,\cdots,d(g,h)\}$ is injective by Remark \ref{rem:two contours}.
\end{proof}

\begin{defn}
Suppose $g,h\in G$ and $p\in \G_{g,h}$. For a subset $\A\subset \C_{p,\lambda}$, we define an order $\le_p$ by defining for $r,s\in\A$,
\[
r\le_p s \iff d(g,(r\cap p)_-) \le d(g,(s\cap p)_-).
\]
\end{defn}

\begin{rem}
The order $\le_p$ is well-defined by Remark \ref{rem:two contours}, that is, $(r\le_p s) \wedge (s\le_p r)$ implies $r=s$. Also, this order is the same as defining
\[
r\le_p s \iff d((r\cap p)_+,h) \ge d((s\cap p)_+,h).
\]
This is used in the proof of Lemma \ref{lem:order independent}.
\end{rem}

Our next goal is to prove Lemma \ref{lem:order independent} where we show that this order is independent of the choice of geodesic $p$. In Lemma \ref{lem:three contours} below and thereafter, given a geodesic $p$ in $\Gamma$ and $r, r' \in \C$, the subgraph $r \cap r' \cap p$ is a subpath of $p$, because $r \cap p$ and $r' \cap p$ are subpaths of $p$ by Lemma \ref{lem:AD} (3).

\begin{lem}
\label{lem:three contours}
Suppose $g,h\in G$, $p\in \G_{g,h}$, and $r_1,r_2,r_3\in \C_{p,2\lambda}$. If $r_1<_p r_2 <_p r_3$, then we have
\[
d(g,(r_1\cap p)_+) < d(g,(r_3\cap p)_-),
\]
in particular, $r_1\cap r_3\cap p=\emptyset$.
\end{lem}

\begin{proof}
We have $d(g,(r_1\cap p)_+) < d(g,(r_2\cap p)_+)$ by $r_1<_p r_2$ and Remark \ref{rem:two contours}. Also, $C'(\lambda)$-condition implies $|r_1 \cap r_2 \cap p|\le |r_1\cap r_2|<\lambda|r_2|$. Hence, we have
\[
d(g,(r_1\cap p)_+)< d(g,(r_2\cap p)_-)+\lambda|r_2|.
\]
By the same argument for $r_2$ and $r_3$, we can also see 
$d(g,(r_2\cap p)_+)-\lambda|r_2|< d(g,(r_3\cap p)_-)$. Note that $r_2\in \C_{p,2\lambda}$ implies $d(g,(r_2\cap p)_-)+2\lambda|r_2|\le d(g,(r_2\cap p)_+)$. Thus, we have
\[
d(g,(r_1\cap p)_+)
< d(g,(r_2\cap p)_-)+\lambda|r_2|
\le d(g,(r_2\cap p)_+)-\lambda|r_2|
< d(g,(r_3\cap p)_-).
\]
\end{proof}

\begin{lem}
\label{lem:CDelata large}
Suppose $g,h\in G$, $x\in X\cup X^{-1}\cup\{1\}$, $p\in \G_{g,h}$ and $q\in \G_{gx,h}$. When $x\in X\cup X^{-1}$, let $\e$ be the edge of $\Gamma$ from $g$ to $gx$ with $\L(\e)=x$ and $\Delta$ be a van Kampen diagram of $\L(\e qp^{-1})$, and when $x=1$, let $\Delta$ be a van Kampen diagram of $\L(qp^{-1})$. Then, we have (cf. Remark \ref{rem:CDelta} for $\C_\Delta$)
\[
\C_\Delta 
\subset \C_{p,\frac{1}{2}-2\lambda}
\cap \C_{q,\frac{1}{2}-2\lambda}.
\]
\end{lem}

\begin{proof}
Let $r\in \C_\Delta$. Since $q$ is geodesic, we have $|r\cap q|\le \frac{1}{2}|r|$. Let $r_-=[(r\cap p)_-,(r\cap q)_-]$ be the subpath of $r$ that doesn't contain $r\cap p$ or $r\cap q$, then $r_-$ is either (i) of length 0, i.e. $(r\cap p)_-=(r\cap q)_-$, (ii) an intersection of $r$ with another contour in $\C_\Delta$, or (iii) coincides $\e$. In any of the cases (i)(ii)(iii), we have $|r_-|<\lambda|r|$. Indeed, this follows from $C'(\lambda)$-condition in case (ii) and from the condition \eqref{eq:*} in case (iii) (Note $|\e|=1<\lambda|r|$). Similarly, letting $r_+=[(r\cap p)_+,(r\cap q)_+]$ be the subpath of $r$ that doesn't contain $r\cap p$ or $r\cap q$, we have $|r_+|<\lambda|r|$. Thus, we have
\[
|r\cap p|
=|r|-|r_-|-|r\cap q|-|r_+|
>\left( 1-\lambda-\frac{1}{2}-\lambda \right)|r|
=\left( \frac{1}{2}-2\lambda \right)|r|,
\]
hence $r\in \C_{p,\frac{1}{2}-2\lambda}$. Since $p$ is geodesic, $r\in \C_{q,\frac{1}{2}-2\lambda}$ also follows similarly.
\end{proof}

\begin{lem}
\label{lem:attached contours}
Suppose $g,h\in G$, $x\in X\cup X^{-1}\cup \{1\}$, $p\in \G_{g,h}$ and $q\in \G_{gx,h}$, then for any $\mu' \in [2\lambda,\frac{1}{2}]$, we have 
\[
\C_{p,\mu'} \subset \C_{q,\mu'-2\lambda}.
\]
\end{lem}

\begin{proof}
When $x\in X\cup X^{-1}$, let $\e$ be the edge of $\Gamma$ from $g$ to $gx$ with $\L(\e)=x$ and $\Delta$ be a van Kampen diagram of $\L(\e qp^{-1})$, and when $x=1$, let $\Delta$ be a van Kampen diagram of $\L(qp^{-1})$. Let $r\in\C_{p,\mu'}$. If $r\in \C_\Delta$, then we have $r\in \C_{q,\frac{1}{2}-2\lambda}$ by Lemma \ref{lem:CDelata large}. Hence, $r\in \C_{q,\mu'-2\lambda}$ by $\mu' \le \frac{1}{2}$.

We assume $r\notin \C_\Delta$ in the following. We define $s_-=\max\{s\in\C_\Delta \mid s<_p r \}$ and $s_+=\min\{s\in\C_\Delta \mid r<_p s \}$. (By convention, we define $s_-,s_+=\emptyset$, if the set defining the maximum and minimum of the right hand side is empty.) (1) When $\Delta$ is the form of Figure \ref{Fig3new}, or the form (a) or (c) of Figure \ref{Fig4new}, we have
$p\sm q=\bigcup_{s\in\C_\Delta}s\cap p$, hence $(r\cap p)\sm q=(r\cap p)\cap \left(\bigcup_{s\in\C_\Delta}s\cap p\right)$. By Remark \ref{rem:two contours} and Lemma \ref{lem:three contours}, we have $(r\cap p)\cap \left(\bigcup_{s\in\C_\Delta}s\cap p\right)=(r \cap s_- \cap p)\sqcup (r \cap s_+ \cap p)$. Thus, by $C'(\lambda)$-condition, we have
\[
|r\cap p|
=|r \cap s_- \cap p|+|(r\cap p)\cap q|+|r\cap s_+ \cap p|
<\lambda|r|+|(r\cap p)\cap q|+\lambda|r|.
\]
This implies
\[
|r\cap q|\ge|(r\cap p)\cap q|>|r\cap p|-2\lambda|r|\ge(\mu'-2\lambda)|r|,
\]
by $r\in \C_{p,\mu'}$. Hence, $r\in \C_{q,\mu'-2\lambda}$. Here, $|r\cap q|>|(r\cap p)\cap q|$ can occur only when $\Delta$ is the form (c) of Figure \ref{Fig4new} and $(r\cap q)_-=gx$.

(2) When $\Delta$ is the form (b) of Figure \ref{Fig4new}, we have
$p\sm q=\e \cup\bigcup_{s\in\C_\Delta}s\cap p$. If $(r\cap p)_-\neq g$, then $|(r\cap p)\cap \e|=0$. This implies $(r\cap p)\sm q = (r\cap p)\cap \left(\bigcup_{s\in\C_\Delta}s\cap p\right)$. Hence, $r\in \C_{q,\mu'-2\lambda}$ follows by the same argument as (1). If $(r\cap p)_-= g$, then $s_-=\emptyset$ and $(r\cap p)\sm q=\e \cup (r\cap s)_+$. Hence, we have
\[
|r\cap p|
=|\e|+|(r\cap p)\cap q|+|r\cap s_+ \cap p|
<\lambda|r|+|(r\cap p)\cap q|+\lambda|r|.
\]
Here, we used $|\e|=1<\lambda|r|$ by condition \eqref{eq:*}. This implies $|r\cap q|=|(r\cap p)\cap q|>|r\cap p|-2\lambda|r|\ge(\mu'-2\lambda)|r|$, hence $r\in \C_{q,\mu'-2\lambda}$.
\end{proof}

\begin{rem}
\label{rem:attached contours}
In Lemma \ref{lem:attached contours}, its proof also shows that for any $r\in\C_{p,\mu'} \sm \C_\Delta$, we have
\[
d((r\cap q)_-,h)> d((r\cap p)_-,h)-\lambda|r|
~~{\rm and}~~
d((r\cap p)_+,h)+\lambda|r|> d((r\cap q)_+,h).
\]
\end{rem}

\begin{cor}
\label{cor:attached contours}
Suppose $g,h\in G$, $x\in X\cup X^{-1}\cup \{1\}$, and $p\in \G_{g,h}$. For any $\mu'\in [2\lambda,\frac{1}{2}]$, we have
\[
\C_{p,\mu'}\subset \C_{g,h,\mu'-2\lambda} \cap \C_{gx,h,\mu'-2\lambda}.
\]
In particular, we have 
$
\C_{g,h,\mu} \cup \C_{gx,h,\mu}
\subset
\C_{g,h,2\lambda} \cap \C_{gx,h,2\lambda}.
$
\end{cor}

\begin{proof}
The first statement follows directly from Lemma \ref{lem:attached contours}. The second statement follows from the first one and $2\lambda\le\mu-2\lambda$.
\end{proof}

\begin{lem}\label{lem:order independent}
Suppose $g,h\in G$, $x\in X\cup X^{-1}\cup \{1\}$, $p\in \G_{g,h}$ and $q\in \G_{gx,h}$. If $r,s\in \C_{p,2\lambda} \cap \C_{q,2\lambda}$, then
\[
r\le_p s \iff r\le_q s.
\]
\end{lem}

\begin{proof}
When $x\in X\cup X^{-1}$, let $\e$ be the edge of $\Gamma$ from $g$ to $gx$ with $\L(\e)=x$ and $\Delta$ be a van Kampen diagram of $\L(\e qp^{-1})$, and when $x=1$, let $\Delta$ be a van Kampen diagram of $\L(qp^{-1})$. If $r,s\in \C_\Delta$, then $r\le_p s \iff r\le_q s$ follows from the van Kampen diagram $\Delta$ (cf. Figure \ref{Fig3new},\ref{Fig4new}).

When $r\in \C_\Delta$, $s\notin \C_\Delta$, and $r<_p s$, we will show $r<_q s$. Let $r'\in \C_\Delta$ be the maximum contour in $(\C_\Delta,\le_p)$ satisfying $r'<_p s$. By maximality of $r'$, we have $(r'\cap p)_+=(r'\cap q)_+$. Indeed, suppose $(r'\cap p)_+\neq(r'\cap q)_+$, then by the van Kampen diagram $\Delta$, there exists $r''\in \C_\Delta$ such that $r'<_p r''$, $(r'\cap p)_+=(r''\cap p)_-$, and $(r'\cap q)_+=(r''\cap q)_-$. By Lemma \ref{lem:three contours}, $(r'\cap p)_+=(r''\cap p)_-$ implies $r''<_p s$, which contradicts maximality of $r'$. By $C'(\lambda)$-condition, we have $d((r'\cap p)_+,h)> d((s\cap p)_-,h)-\lambda|s|$. We also have $d((s\cap p)_-,h) \ge d((s\cap p)_+,h)+2\lambda|s|$ by $s\in \C_{p,2\lambda}$, and $d((s\cap p)_+,h)+\lambda|s|> d((s\cap q)_+,h)$ by Remark \ref{rem:attached contours}. Thus, we have
\begin{equation}
\label{eq:pq}
\begin{split}
    d((r'\cap q)_+,h)
=d((r'\cap p)_+,h)
&>d((s\cap p)_-,h)-\lambda|s| \\
&\ge d((s\cap p)_+,h)+2\lambda|s|-\lambda|s|
=d((s\cap p)_+,h)+\lambda|s| \\
&>d((s\cap q)_+,h).
\end{split}
\end{equation}
This implies
\[
d(gx,(r'\cap q)_+)
=d(gx,h)-d((r'\cap q)_+,h)
<d(gx,h)-d((s\cap q)_+,h)
=d(gx,(s\cap q)_+).
\]
Hence $r'<_q s$ by Remark \ref{rem:two contours}. As we saw above, we have $r<_q r'$ by $r,r'\in \C_\Delta$, thus, $r<_q r' <_q s$. Similarly, we can show that $r<_q s$ implies $r<_p s$. Hence, when $r\in \C_\Delta$ and $s\notin \C_\Delta$, we have $r<_p s \iff r<_q s$. By similar argument, we can also show $r<_p s \iff r<_q s$ when $r\notin \C_\Delta$ and $s\in \C_\Delta$.

Finally, when $r,s\notin \C_\Delta$, we will show $r<_p s$ implies $r<_q s$. If there exists $t\in\C_\Delta$ such that $r<_p t <_p s$, then we have $r<_q t <_q s$ by the above argument. If there is no $t\in\C_\Delta$ satisfying $r<_p t <_p s$, then we have $(r\cap p)_+=(r\cap q)_+$. Indeed, suppose $(r\cap p)_+\neq (r\cap q)_+$, then there exists $t\in \C_\Delta$ such that $r<_p t$, $(t\cap p)_-=(t\cap q)_-$, and $d(g,(t\cap p)_-)\le d(g,(r\cap p)_+)$. By Lemma \ref{lem:three contours}, $d(g,(t\cap p)_-)\le d(g,(r\cap p)_+)$ implies $t<_p s$, which contradicts our assumption. Thus, by the same argument as the inequalities in $\eqref{eq:pq}$, we can show $d((r\cap q)_+,h)>d((s\cap q)_+,h)$, which implies $r<_q s$. Similarly, we can show that $r<_q s$ implies $r<_p s$. Hence, when $r,s\notin \C_\Delta$, we have $r<_p s \iff r<_q s$.
\end{proof}

\begin{rem}
\label{rem:order}
For simplicity, for $g,h\in G$, $p\in\G_{g,h}$, and a subset $\A=\{r_1,r_2,\cdots,r_n\}$ of $\C_{p,\lambda}$ satisfing $r_1<_p r_2 <_p \cdots <_p r_n$, we denote
\[
\A=\{r_1<_p r_2 <_p \cdots <_p r_n\}.
\]
Also, by Lemma \ref{lem:order independent}, for $g,h\in G$, $x\in X\cup X^{-1}$ and a subset $\A=\{r_1,r_2,\cdots,r_n\}$ of $\C_{g,h,2\lambda}\cap \C_{gx,h,2\lambda}$ (cf. \eqref{eq:all intersection}) satisfying $r_1<_p r_2 <_p \cdots <_p r_n$ for some (equivalently any) $p\in \G_{g,h}\cup \G_{gx,h}$, we denote
\[
\A=\{r_1 < r_2 < \cdots < r_n\}.
\]
\end{rem}

Now, we are ready to define a function on $\C$, given a pair of vertices (cf. Definition \ref{def:xi[g,h,psi]}). First, we assign several real numbers ($\alpha,\beta,\rho,\sigma,\tau$) to aligned contours on a geodesic.

\begin{defn}
\label{def:alpha beta gamma}
Suppose $g,h\in G$, $p\in \G_{g,h}$, and $\A=\{r_1 <_p r_2 <_p \cdots <_p r_n\}$ is a subset of $\C_{p,2\lambda}$. We define for each $i\in\{1,\cdots,n\}$,
\[
\alpha_{r_i,p,\A}=\frac{|r_{i-1} \cap r_i \cap p|}{|r_i|},
\]
\[
\beta_{r_i,p,\A}=\frac{|r_i\cap r_{i+1} \cap p|}{|r_i|}.
\]
By convention, we consider $r_0=r_{n+1}=\emptyset$ and define $\alpha_{r_1,p,\A}=\beta_{r_n,p,\A}=0$.
\end{defn}

\begin{rem}
\label{rem:alpha beta}
\begin{itemize}
\item[(1)]
For any $i\in\{1,\cdots,n-1\}$, we have
\[
\beta_{r_i,p,\A}|r_i|
=|r_i\cap r_{i+1} \cap p|
=\alpha_{r_{i+1},p,\A}|r_{i+1}|.
\]
\item[(2)]
For any $i\in\{1,\cdots,n\}$, by $C'(\lambda)$-condition, we have
\[
\alpha_{r_i,p,\A}<\lambda~~{\rm and}~~\beta_{r_i,p,\A}<\lambda.
\]
\end{itemize}
\end{rem}

Throughout Section \ref{sec:construction of arrays}, let $\nu_0,\nu_1$ be two real numbers satisfying
\begin{equation}\label{eq:1nu}
    \mu+2\lambda \le \nu_0 < \nu_1 \le \frac{1}{2}-4\lambda
\end{equation}
and
\begin{equation}\label{eq:2nu}
    \nu_0+\lambda<\nu_1,
\end{equation}
and we define a function $\psi\colon\RR\to\RR$ by
\begin{equation}
\label{eq:psi}
    \psi(x)=
    \begin{cases}
    0 
    & {\rm if~}x\le \nu_0\\
    \frac{1}{\nu_1-\nu_0}(x-\nu_0) 
    &{\rm if~}\nu_0<x<\nu_1 \\
    1 
    & {\rm if~}x\ge \nu_1.
\end{cases}
\end{equation}
The real numbers $\nu_0,\nu_1$ satisfying \eqref{eq:1nu} and \eqref{eq:2nu} exist by $\lambda = \frac{1}{33}$ and $\mu = \frac{4}{33} = 4\lambda$. In Section \ref{sec:main thm}, we will use two different choices for the pair $(\nu_0,\nu_1)$, and these will be denoted $(\nu_{1,0},\nu_{1,1})$ and $(\nu_{2,0},\nu_{2,1})$.

We define a real number $\rho_{r_i,p,\A}\in[0,1]$ for each $i\in\{1\cdots,n\}$ inductively from $1$ to $n$ as follows.
\[
\rho_{r_1,p,\A}=\psi\left( \frac{|r_1\cap p|}{|r_1|} \right),
\]
\[
\rho_{r_{i+1},p,\A}
=\psi\left( \frac{|r_{i+1}\cap p|}{|r_{i+1}|}-\rho_{r_i,p,\A}\alpha_{r_{i+1},p,\A} \right).
\]
Just for the sake of notation, we define $\rho_{r_0,p,\A}=0$. Similarly, we define a real number $\sigma_{r_i,p,\A}\in[0,1]$ for each $i\in\{1\cdots,n\}$ inductively from $n$ to $1$ as follows.
\[
\sigma_{r_n,p,\A}=\psi\left( \frac{|r_n\cap p|}{|r_n|} \right),
\]
\[
\sigma_{r_{i-1},p,\A}
=\psi\left( \frac{|r_{i-1}\cap p|}{|r_{i-1}|}-\sigma_{r_i,p,\A}\beta_{r_{i-1},p,\A} \right).
\]
Again for the sake of notation, we define $\sigma_{r_{n+1},p,\A}=0$.

The function $\psi$ in \eqref{eq:psi} have two key properties. First, when a contour $r_i \in \A$ has small intersection with the geodesic $p$ (i.e. $\frac{|r_i\cap p|}{|r_i|}$ is small), the condition $\forall\,x \le v_0,\, \psi(x) = 0$ eliminates the effect of $r_i$ on the real numbers $\{\rho_{r_j,p,\A}, \sigma_{r_j,p,\A}\}_{j=1}^n$ defined above. That is, we get the same numbers for $\{\rho_{r_j,p,\A}, \sigma_{r_j,p,\A}\}_{j=1}^n$ even if we remove $r_i$ from $\A$. Secondly, when $r_i \in \A$ has large intersection with $p$, $r_i$ can be a reason of two branching-out geodesic paths as in Figure \ref{Fig4new}. In this case, the condition $\forall\,x \ge v_1,\, \psi(x) = 1$ eliminates the difference of lengths of intersections of $r_i$ with the two geodesics.

The idea of introducing $\{\rho_{r_j,p,\A}, \sigma_{r_j,p,\A}\}_{j=1}^n$ is to make $\{\tau_{r_j,p,\A}\}_{j=1}^n$ in Definition \ref{def:xi[p,psi,A]} stable under branching-out of geodesic paths. Also, $\{\rho_{r_j,p,\A}\}_{j=1}^n$ exponentially attenuate the change of $\frac{|r_1\cap p|}{|r_1|}$, which is made by moving the endpoint $g$, as $j$ goes from $1$ to $n$. Similarly, $\{\sigma_{r_j,p,\A}\}_{j=1}^n$ attenuate the change of $\frac{|r_n\cap p|}{|r_n|}$ made by moving the other endpoint $h$ as $j$ goes from $n$ to $1$. These points will become clear in the proof of Lemma \ref{lem:xi estimate}.

\begin{defn}
\label{def:xi[p,psi,A]}
For $g,h\in G$, $p\in\G_{g,h}$, the function $\psi$ in \eqref{eq:psi}, and a subset $\A=\{r_1 <_p r_2 <_p \cdots <_p r_n\}$ of $\C_{p,2\lambda}$, we define a map $\xi[p,\psi,\A] \colon \C \to \RR$ by
\[
\xi[p,\psi,\A]
=\sum_{i=1}^n \psi\left( \frac{|r_i\cap p|}{|r_i|}-\rho_{r_{i-1},p,\A}\alpha_{r_i,p,\A}-\sigma_{r_{i+1},p,\A}\beta_{r_i,p,\A} \right)|r_i|\cdot1_{r_i},
\]
For simplicity, we will denote
\begin{equation*}
    \tau_{r_i,p,\A}=\frac{|r_i\cap p|}{|r_i|}-\rho_{r_{i-1},p,\A}\alpha_{r_i,p,\A}-\sigma_{r_{i+1},p,\A}\beta_{r_i,p,\A}.
\end{equation*}
\end{defn}

\begin{rem}
\label{rem:tau}
For any $i\in\{1,\cdots,n\}$, we have
\[
\frac{|r_i\cap p|}{|r_i|}-2\lambda
< \tau_{r_i,p,\A}
\le \frac{|r_i\cap p|}{|r_i|}.
\]
The first inequality follows from Remark \ref{rem:alpha beta} (2).
\end{rem}

\begin{defn}
\label{def:xi[g,h,psi]}
For $g,h\in G$ and the function $\psi$ in \eqref{eq:psi}, we define a map $\xi[g,h,\psi] \colon \C \to \RR$ by
\[
\xi[g,h,\psi](r)=\max_{p\in \G_{g,h}}\xi[p,\psi,\C_{g,h,\mu}](r)
\]
for each $r\in \C$.
\end{defn}

\begin{rem}
In Definition \ref{def:xi[g,h,psi]}, $\xi[p,\psi,\C_{g,h,\mu}]$ is well-defined since $\C_{g,h,\mu}$ is a finite subset of $\C_{p,2\lambda}$. Also, the maximum over $\G_{g,h}$ exists, because $\G_{g,h}$ is finite since $\Gamma$ is locally finite.
\end{rem}

\begin{rem}\label{rem:xi positive}
For any $r\in\C$, we have $0\le \xi[g,h,\psi](r) \le |r|$.
\end{rem}

The following is straightforward.

\begin{lem}
\label{lem:xi sym eqiv}
\begin{itemize}
    \item[(1)]
    $\xi[g,h,\psi]=\xi[h,g,\psi]$ for any $g,h\in G$.
    \item[(2)]
    $\xi[kg,kh,\psi]=\pi_\C(k)\xi[h,g,\psi]$ for any $g,h,k\in G$.
\end{itemize}
\end{lem}

\begin{proof}
(1) follows since we have $\C_{g,h,\mu}=\C_{h,g,\mu}$ and we have for any $p\in \G_{g,h}$ and $r\in \C_{g,h,\mu}$, $\alpha_{r,p,\C_{g,h,\mu}}=\beta_{r,p^{-1},\C_{h,g,\mu}}$ and $\beta_{r,p,\C_{g,h,\mu}}=\alpha_{r,p^{-1},\C_{h,g,\mu}}$, which imply $\rho_{r,p,\C_{g,h,\mu}}=\sigma_{r,p^{-1},\C_{h,g,\mu}}$ and $\sigma_{r,p,\C_{g,h,\mu}}=\rho_{r,p^{-1},\C_{h,g,\mu}}$, hence $\tau_{r,p,\C_{g,h,\mu}}=\tau_{r,p^{-1},\C_{h,g,\mu}}$.

(2) follows since we have $\C_{kg,kh,\mu}=k\C_{g,h,\mu}$ and we have for any $p\in \G_{g,h}$ and $r\in \C_{g,h,\mu}$, $\alpha_{kr,kp,\C_{kg,kh,\mu}}=\alpha_{r,p,\C_{g,h,\mu}}$ and $\beta_{kr,kp,\C_{kg,kh,\mu}}=\beta_{r,p,\C_{g,h,\mu}}$, which imply $\rho_{kr,kp,\C_{kg,kh,\mu}}=\rho_{r,p,\C_{g,h,\mu}}$ and $\sigma_{kr,kp,\C_{kg,kh,\mu}}=\sigma_{r,p,\C_{g,h,\mu}}$, hence $\tau_{kr,kp,\C_{kg,kh,\mu}}=\tau_{r,p,\C_{g,h,\mu}}$.
\end{proof}

\begin{lem}
\label{lem:xi for big contour}
Suppose $g,h\in G$ and $p\in\G_{g,h}$, then for any $r\in \C_{p,\nu_1+2\lambda}$, we have
\[
\xi[g,h,\psi](r)=|r|.
\]
\end{lem}

\begin{proof}
Since $r\in \C_{p,\nu_1+2\lambda}$, we have for any $q\in\G_{g,h}$, $r\in \C_{q,\nu_1}$ by Lemma \ref{lem:attached contours}. By $\mu \le \nu_1$, this implies $r\in\C_{g,h,\mu}$. Also, since $r\in \C_{p,\nu_1+2\lambda}$, we have $|r\cap p|\ge (\nu_1+2\lambda)|r|$. Hence, by Remark \ref{rem:tau}, we have
\[
\tau_{r,p,\C_{g,h,\mu}}>\frac{|r\cap p|}{|r|}-2\lambda\ge \nu_1+2\lambda-2\lambda=\nu_1.
\]
This implies $\xi[p,\psi,\C_{g,h,\mu}](r)=\psi(\tau_{r,p,\C_{g,h,\mu}})|r|=1\cdot|r|$, hence
\[
|r|=\xi[p,\psi,\C_{g,h,\mu}](r)\le \xi[g,h,\psi](r)\le|r|.
\]
\end{proof}

Lemma \ref{lem:xi stable} means that $\xi[p,\psi,\C_{g,h,\mu}]$ is unchanged even though we add contours that are not attached to a geodesic big enough. Recall $\C_{g,h,\mu} \cup \C_{gx,h,\mu} \subset \C_{g,h,2\lambda} \cap \C_{gx,h,2\lambda}$ by Corollary \ref{cor:attached contours}, hence $\xi[p,\psi,\C_{g,h,\mu}\cup\C_{gx,h,\mu}]$ below is well-defined.

\begin{lem}\label{lem:xi stable}
For any $g,h\in G$, $p\in \G_{g,h}$, and $x\in X\cup X^{-1}$, we have
\[
\xi[p,\psi,\C_{g,h,\mu}]=\xi[p,\psi,\C_{g,h,\mu}\cup\C_{gx,h,\mu}].
\]
\end{lem}

\begin{proof}
Just to simplify notations, we denote $\A=\C_{g,h,\mu}\cup\C_{gx,h,\mu}$. For any $r\in \C \sm \A$, we have $\xi[p,\psi,\C_{g,h,\mu}](r)=0=\xi[p,\psi,\A](r)$. For any $r\in \A \sm \C_{g,h,\mu}$, we have $r\notin \C_{p,\mu+2\lambda}$ by Corollary \ref{cor:attached contours}, hence
\[
|r\cap p|<(\mu+2\lambda)|r|\le \nu_0|r|,
\]
since we have $\mu+2\lambda\le \nu_0$ by \eqref{eq:1nu}. By Remark \ref{rem:tau}, this implies
$\tau_{r,p,\A}\le \nu_0$, hence $\psi(\tau_{r,p,\A})=0$. Thus, we have $\xi[p,\psi,\A](r)=0=\xi[p,\psi,\C_{g,h,\mu}](r)$.

Now we will show $\xi[p,\psi,\A](r)=\xi[p,\psi,\C_{g,h,\mu}](r)$ for any $r\in\C_{g,h,\mu}$. Let $\A=\{r_1<_p\cdots <_p r_n\}$, then $\C_{g,h,\mu}=\{r_{i_1}<_p\cdots <_p r_{i_m}\}$ is a subsequence of $\A$. We show $\rho_{r_{i_k},p,\A}=\rho_{r_{i_k},p,\C_{g,h,\mu}}$ for any $k\in\{1,\cdots,m\}$ by induction from $1$ to $m$. Since $r_{i_1-1}\notin \C_{g,h,\mu}$, we have $|r_{i_1-1}\cap p|< \nu_0|r_{i_1-1}|$ as above. This implies $\frac{|r_{i_1-1}\cap p|}{|r_{i_1-1}|} -\rho_{r_{i_1-2},p,\A}\alpha_{r_{i_1-1},p,\A} \le \nu_0$, hence $\rho_{r_{i_1-1},p,\A}=0$. Thus, we have
\[
\rho_{r_{i_1},p,\A}
=\psi\left(\frac{|r_{i_1}\cap p|}{|r_{i_1}|}-0\right)
=\rho_{r_{i_1},p,\C_{g,h,\mu}}.
\]
Next, assume $\rho_{r_{i_k},p,\A}=\rho_{r_{i_k},p,\C_{g,h,\mu}}$. If $i_k+1=i_{k+1}$, then we have $\alpha_{r_{i_{k+1}},p,\A}=\alpha_{r_{i_{k+1}},p,\C_{g,h,\mu}}$, hence,
\begin{align*}
    \rho_{r_{i_{k+1}},p,\A}
&=\psi\left( \frac{|r_{i_{k+1}}\cap p|}{|r_{i_{k+1}}|} -\rho_{r_{i_k},p,\A}\alpha_{r_{i_{k+1}},p,\A} \right)\\
&=\psi\left( \frac{|r_{i_{k+1}}\cap p|}{|r_{i_{k+1}}|} -\rho_{r_{i_k},p,\C_{g,h,\mu}}\alpha_{r_{i_{k+1}},p,\C_{g,h,\mu}} \right)
=\rho_{r_{i_{k+1}},p,\C_{g,h,\mu}}.
\end{align*}
If $i_k+1<i_{k+1}$, then $r_{i_{k+1}-1}\notin \C_{g,h,\mu}$. Hence, we have $\rho_{r_{i_{k+1}-1},p,\A}=0$ as above. On the other hand, by $r_{i_k}<_p r_{i_k+1}<_p r_{i_{k+1}}$ and Lemma \ref{lem:three contours}, we have $\alpha_{r_{i_{k+1}},p,\C_{g,h,\mu}}=0$. Hence, we have
\begin{align*}
    \rho_{r_{i_{k+1}},p,\A}
&=\psi\left(\frac{|r_{i_{k+1}}\cap p|}{|r_{i_{k+1}}|}-0\cdot\alpha_{r_{i_{k+1}},p,\A} \right)\\
&=\psi\left(\frac{|r_{i_{k+1}}\cap p|}{|r_{i_{k+1}}|}-\rho_{r_{i_k},p,\C_{g,h,\mu}}\cdot0 \right)
=\rho_{r_{i_{k+1}},p,\C_{g,h,\mu}}.
\end{align*}
Thus, $\rho_{r_{i_k},p,\A}=\rho_{r_{i_k},p,\C_{g,h,\mu}}$ for any $k\in\{1,\cdots,m\}$. In the same way by induction from $m$ to $1$, we can also see $\sigma_{r_{i_k},p,\A}=\sigma_{r_{i_k},p,\C_{g,h,\mu}}$ for any $k\in\{1,\cdots,m\}$. We claim $\rho_{r_{i_k-1},p,\A}\alpha_{r_{i_k},p,\A}=\rho_{r_{i_{k-1}},p,\C_{g,h,\mu}}\alpha_{r_{i_k},p,\C_{g,h,\mu}}$ for any $k\in\{1,\cdots,m\}$. Indeed, if $i_{k-1}=i_k-1$, then this follows from $\alpha_{r_{i_k},p,\A}=\alpha_{r_{i_k},p,\C_{g,h,\mu}}$ and $\rho_{r_{i_{k-1}},p,\A}=\rho_{r_{i_{k-1}},p,\C_{g,h,\mu}}$ as we have shown. If $i_{k-1}<i_k-1$, then $r_{i_k-1}\notin\C_{g,h,\mu}$, hence we have $\rho_{r_{i_k-1},p,\A}=0$ and $\alpha_{r_{i_k},p,\C_{g,h,\mu}}=0$ as above. This implies $\rho_{r_{i_k-1},p,\A}\alpha_{r_{i_k},p,\A}=0=\rho_{r_{i_{k-1}},p,\C_{g,h,\mu}}\alpha_{r_{i_k},p,\C_{g,h,\mu}}$. Similarly, we can also see $\sigma_{r_{i_k+1},p,\A}\beta_{r_{i_k},p,\A}=\sigma_{r_{i_{k+1}},p,\C_{g,h,\mu}}\beta_{r_{i_k},p,\C_{g,h,\mu}}$ for any $k\in\{1,\cdots,m\}$. Thus, for any $k\in\{1,\cdots,m\}$, we have $\tau_{r_{i_k},p,\A}=\tau_{r_{i_k},p,\C_{g,h,\mu}}$, hence $\xi[p,\psi,\A](r_{i_k})=\xi[p,\psi,\C_{g,h,\mu}](r_{i_k})$.
\end{proof}

Throughout Section \ref{sec:construction of arrays}, we define
\begin{equation}\label{eq:K}
    K=\frac{\lambda}{\nu_1-\nu_0}
\end{equation}
for simplicity. We have $0<K<1$ by $\nu_0+\lambda<\nu_1$.

\begin{lem}
\label{lem:xi estimate}
Suppose $g,h\in G$, $x\in X\cup X^{-1}$ and $\C_{g,h,\mu}\cup\C_{gx,h,\mu}=\{r_1<\cdots<r_n\}$, then for any $p\in \G_{g,h}$, there exists $q\in \G_{gx,h}$ such that for any $i\in\{1,\cdots,n\}$,
\begin{equation}
\label{eq:xi}
    |\xi[p,\psi,\C_{g,h,\mu}\cup\C_{gx,h,\mu}](r_i)-\xi[q,\psi,\C_{g,h,\mu}\cup\C_{gx,h,\mu}](r_i)|
<\frac{K^{i-1}}{\nu_1-\nu_0}.
\end{equation}
\end{lem}

\begin{proof}
For simplicity, we denote $\A=\C_{g,h,\mu}\cup\C_{gx,h,\mu}$. Let $p\in \G_{g,h}$ and fix $q_0\in \G_{gx,h}$. Also, let $\e$ be the edge of $\Gamma$ from $g$ to $gx$ with $\L(\e)=x$.

(1) When $gx\in p$, we define $q=p_{[gx,h]}$ (cf. Definition \ref{def:Cayley} for $p_{[gx,h]}$) and when $g\in q_0$, we define $q=\e^{-1} p$. In both cases, we have $q\in \G_{gx,h}$, $(r_1\cap p)\Delta (r_1\cap q)\subset \e$ ($\Delta$ means symmetric difference here), and for any $i\in\{2,\cdots,n\}$, $r_i\cap p=r_i\cap q$. This implies for any $i\in\{1,\cdots,n\}$, $\alpha_{r_i,p,\A}=\alpha_{r_i,q,\A}$ and $\beta_{r_i,p,\A}=\beta_{r_i,q,\A}$. Also, by induction from $n$ to $2$, we can see for any $i\in\{2,\cdots,n\}$, $\sigma_{r_i,p,\A}=\sigma_{r_i,q,\A}$. We will show for any $i\in\{1,\cdots,n\}$,
\begin{equation}
\label{eq:rho}
    |\rho_{r_i,p,\A}-\rho_{r_i,q,\A}||r_i|<\frac{K^{i-1}}{\nu_1-\nu_0}
\end{equation}
by induction from $1$ to $n$. For $r_1$, by $||r_1\cap p|-|r_1\cap q|| \le |\e| = 1$, we have
\begin{align*}
    |\rho_{r_1,p,\A}-\rho_{r_1,q,\A}||r_1|
    &=\left| \psi\left( \frac{|r_1\cap p|}{|r_1|} \right)-\psi\left( \frac{|r_1\cap q|}{|r_1|} \right) \right||r_1| \\
    &\le \frac{1}{\nu_1-\nu_0} \left| \frac{|r_1\cap p|}{|r_1|} -\frac{|r_1\cap q|}{|r_1|} \right||r_1| \\
    &=\frac{1}{\nu_1-\nu_0}||r_1\cap p|-|r_1\cap q|| \\
    &\le \frac{1}{\nu_1-\nu_0}\cdot 1.
\end{align*}

Next, assume \eqref{eq:rho} holds for $r_i$ with $i\ge1$, then by Remark \ref{rem:alpha beta} (1)(2), we have
\begin{align*}
    &|\rho_{r_{i+1},p,\A}-\rho_{r_{i+1},q,\A}|\cdot|r_{i+1}|\\
    &= \left| \psi\left( \frac{|r_{i+1}\cap p|}{|r_{i+1}|}-\rho_{r_i,p,\A}\alpha_{r_{i+1},p,\A} \right) - \psi\left( \frac{|r_{i+1}\cap q|}{|r_{i+1}|}-\rho_{r_i,q,\A}\alpha_{r_{i+1},q,\A} \right) \right|\cdot|r_{i+1}| \\
    &\le \frac{1}{\nu_1-\nu_0}\cdot \left| \left( \frac{|r_{i+1}\cap p|}{|r_{i+1}|}-\rho_{r_i,p,\A}\alpha_{r_{i+1},p,\A} \right) - \left( \frac{|r_{i+1}\cap q|}{|r_{i+1}|}-\rho_{r_i,q,\A}\alpha_{r_{i+1},q,\A} \right) \right|\cdot|r_{i+1}| \\
    &=\frac{1}{\nu_1-\nu_0}\cdot |-\rho_{r_i,p,\A}\alpha_{r_{i+1},p,\A} +\rho_{r_i,q,\A}\alpha_{r_{i+1},q,\A}|\cdot|r_{i+1}| \\
    &=\frac{1}{\nu_1-\nu_0}\cdot |-\rho_{r_i,p,\A} +\rho_{r_i,q,\A}|\cdot \alpha_{r_{i+1},p,\A}\cdot|r_{i+1}|\\
    &=\frac{1}{\nu_1-\nu_0}\cdot |-\rho_{r_i,p,\A} +\rho_{r_i,q,\A}|\cdot \beta_{r_i,p,\A}\cdot|r_i|\\
    &<\frac{1}{\nu_1-\nu_0}\cdot |-\rho_{r_i,p,\A} +\rho_{r_i,q,\A}|\cdot\lambda\cdot|r_i|
    =\frac{\lambda}{\nu_1-\nu_0} \cdot|\rho_{r_i,p,\A} -\rho_{r_i,q,\A}|\cdot|r_i|\\
    &<\frac{\lambda}{\nu_1-\nu_0}\cdot \frac{K^{i-1}}{\nu_1-\nu_0}
    =K\cdot\frac{K^{i-1}}{\nu_1-\nu_0}.
\end{align*}
Here, we also used $|r_{i+1}\cap p|=|r_{i+1}\cap q|$ and $\alpha_{r_{i+1},p,\A}=\alpha_{r_{i+1},q,\A}$. Thus, \eqref{eq:rho} holds for any $i\in\{1,\cdots,n\}$ by induction. Since $\sigma_{r_{i+1},p,\A}\beta_{r_i,p,\A}=\sigma_{r_{i+1},q,\A}\beta_{r_i,q,\A}$ for any $i\in\{1,\cdots,n\}$, this implies by the same calculation as above,
\begin{align*}
    &|\xi[p,\psi,\A](r_1)-\xi[q,\psi,\A](r_1)|\\
    &=\left| \psi\left( \frac{|r_1\cap p|}{|r_1|}-\sigma_{r_2,p,\A}\beta_{r_1,p,\A} \right)-\psi\left( \frac{|r_1\cap q|}{|r_1|}-\sigma_{r_2,q,\A}\beta_{r_1,q,\A} \right) \right|\cdot|r_1|\\
    &\le \frac{1}{\nu_1-\nu_0}\cdot\left| \left( \frac{|r_1\cap p|}{|r_1|}-\sigma_{r_2,p,\A}\beta_{r_1,p,\A} \right)-\left( \frac{|r_1\cap q|}{|r_1|}-\sigma_{r_2,q,\A}\beta_{r_1,q,\A} \right) \right|\cdot|r_1|\\
    &= \frac{1}{\nu_1-\nu_0}\cdot\left| \left( \frac{|r_1\cap p|}{|r_1|} \right)-\left( \frac{|r_1\cap q|}{|r_1|} \right) \right|\cdot|r_1|\\
    &\le \frac{1}{\nu_1-\nu_0}\cdot 1
\end{align*}
and for any $i\in\{2,\cdots,n\}$,
\begin{align*}
    &|\xi[p,\psi,\A](r_i)-\xi[q,\psi,\A](r_i)|\\
    &=\bigg| \psi\bigg( \frac{|r_i\cap p|}{|r_i|}-\rho_{r_{i-1},p,\A}\alpha_{r_i,p,\A}-\sigma_{r_{i+1},p,\A}\beta_{r_i,p,\A} \bigg)\\
    &~~~~~~~~~~~~~~~ -\psi\bigg( \frac{|r_i\cap q|}{|r_i|}-\rho_{r_{i-1},q,\A}\alpha_{r_i,q,\A}-\sigma_{r_{i+1},q,\A}\beta_{r_i,q,\A} \bigg) \bigg|\cdot|r_i|\\
    &\le \frac{1}{\nu_1-\nu_0}\cdot \bigg| \bigg( \frac{|r_i\cap p|}{|r_i|}-\rho_{r_{i-1},p,\A}\alpha_{r_i,p,\A}-\sigma_{r_{i+1},p,\A}\beta_{r_i,p,\A} \bigg)\\
    &~~~~~~~~~~~~~~~~~~~~~~~~-\bigg( \frac{|r_i\cap q|}{|r_i|}-\rho_{r_{i-1},q,\A}\alpha_{r_i,q,\A}-\sigma_{r_{i+1},q,\A}\beta_{r_i,q,\A} \bigg) \bigg|\cdot|r_i|\\
    &=\frac{1}{\nu_1-\nu_0}\cdot \left| \left( \frac{|r_i\cap p|}{|r_i|}-\rho_{r_{i-1},p,\A}\alpha_{r_i,p,\A} \right)-\left( \frac{|r_i\cap q|}{|r_i|}-\rho_{r_{i-1},q,\A}\alpha_{r_i,q,\A} \right) \right|\cdot|r_i|\\
    &<\frac{\lambda}{\nu_1-\nu_0}\cdot |\rho_{r_{i-1},p,\A} -\rho_{r_{i-1},q,\A}|\cdot|r_{i-1}|\\
    &<\frac{\lambda}{\nu_1-\nu_0}\cdot \frac{K^{i-2}}{\nu_1-\nu_0}
    =\frac{K^{i-1}}{\nu_1-\nu_0}.
\end{align*}
Thus, \eqref{eq:xi} holds for $q=p_{[gx,h]}$ when $gx\in p$ and $q=\e^{-1} p$ when $g\in q_0$.

(2) When $gx\notin p$ and $g\notin q_0$, we define $v\in G$ to be the furthest vertex from $h$ in $p\cap q_0$ and define $q=q_{0[gx,v]}p_{[v,h]}$. We have $q\in \G_{gx,h}$. Let $\Delta$ be a van Kampen diagram of $\L(\e qp^{-1})$, then since $\Delta$ has the form (a) of Figure \ref{Fig4new}, there exists $k\in \{1,\cdots,n\}$ such that $\C_\Delta=\{r_1<\cdots<r_k\}$ by Lemma \ref{lem:CDelata large}, Corollary \ref{cor:attached contours}, and $\mu \le \frac{1}{2}-4\lambda$. Also, we have $r_i\cap p = r_i\cap q$ and $\alpha_{r_i,p,\A}=\alpha_{r_i,q,\A}$ for any $i\in\{k+2,\cdots,n\}$, and $\beta_{r_i,p,\A}=\beta_{r_i,q,\A}$ for any $i\in\{k+1,\cdots,n\}$. Hence, by induction from $n$ to $k+2$, we can see $\sigma_{r_i,p,\A}=\sigma_{r_i,q,\A}$ for any $i\in\{k+2,\cdots,n\}$.

We claim $\rho_{r_i,p,\A}=\rho_{r_i,q,\A}$ for any $i\in\{k+1,\cdots,n\}$. Indeed, we have $|r_k\cap p|\ge \left(\frac{1}{2}-2\lambda\right)|r_k|$ by $r_k\in \C_\Delta$ and Lemma \ref{lem:CDelata large}. By Remark \ref{rem:alpha beta} (2), we also have $\alpha_{r_k,p,\A}<\lambda$, hence
\[
\frac{|r_k\cap p|}{|r_k|}-\rho_{r_{k-1},p,\A}\alpha_{r_k,p,\A}>\frac{1}{2}-2\lambda-\lambda\ge \nu_1.
\]
This implies $\rho_{r_k,p,\A}=1$, hence 
\begin{equation}
\label{eq:k+1}
    \frac{|r_{k+1}\cap p|}{|r_{k+1}|}-\rho_{r_k,p,\A}\alpha_{r_{k+1},p,\A}
=\frac{|r_{k+1}\cap p|}{|r_{k+1}|}-\frac{|r_k\cap r_{k+1} \cap p|}{|r_{k+1}|}
=\frac{|p_{[v,(r_{k+1}\cap p)_+]}|}{|r_{k+1}|}.
\end{equation}
Similarly, we have $\frac{|r_{k+1}\cap q|}{|r_{k+1}|}-\rho_{r_k,q,\A}\alpha_{r_{k+1},q,\A}=\frac{|q_{[v,(r_{k+1}\cap q)_+]}|}{|r_{k+1}|}$. Here, since $q_{[v,h]}=p_{[v,h]}$, we have $(r_{k+1}\cap p)_+=(r_{k+1}\cap q)_+$ and $p_{[v,(r_{k+1}\cap p)_+]}=q_{[v,(r_{k+1}\cap q)_+]}$. This implies
\begin{equation}
\label{eq:rho k+1}
    \rho_{r_{k+1},p,\A}
    =\psi\left( \frac{|p_{[v,(r_{k+1}\cap p)_+]}|}{|r_{k+1}|} \right)
    =\psi\left( \frac{|q_{[v,(r_{k+1}\cap q)_+]}|}{|r_{k+1}|} \right)
    =\rho_{r_{k+1},q,\A}.
\end{equation}
Since $r_i\cap p = r_i\cap q$ and $\alpha_{r_i,p,\A}=\alpha_{r_i,q,\A}$ for any $i\in\{k+2,\cdots,n\}$, we get $\rho_{r_i,p,\A}=\rho_{r_i,q,\A}$ for any $i\in\{k+1,\cdots,n\}$ by induction from $k+1$ to $n$. Finally, we will see for any $i\in\{1,\cdots,n\}$,
\begin{equation}
\label{eq:xi equal}
    \xi[p,\psi,\A](r_i)=\xi[q,\psi,\A](r_i).
\end{equation}
Indeed, this follows for $i\in\{k+2,\cdots,n\}$, since we have $\tau_{r_i,p,\A}=\tau_{r_i,q,\A}$ by what we have shown. For $r_{k+1}$, we can see $\frac{|r_{k+1}\cap p|}{|r_{k+1}|}-\rho_{r_k,p,\A}\alpha_{r_{k+1},p,\A}=\frac{|r_{k+1}\cap q|}{|r_{k+1}|}-\rho_{r_k,q,\A}\alpha_{r_{k+1},q,\A}$ by the same calculation as \eqref{eq:k+1} and \eqref{eq:rho k+1}. We also have $\sigma_{r_{k+2},p,\A}=\sigma_{r_{k+2},q,\A}$ and $\beta_{r_{k+1},p,\A}=\beta_{r_{k+1},q,\A}$. This implies $\tau_{r_{k+1},p,\A}=\tau_{r_{k+1},q,\A}$, hence \eqref{eq:xi equal}. For $i\in\{1,\cdots,k\}$, we have $|r_i\cap p|\ge \left(\frac{1}{2}-2\lambda\right)|r_i|$ by $r_i\in \C_\Delta$ and Lemma \ref{lem:CDelata large}. By Remark \ref{rem:alpha beta} (2), we also have $\alpha_{r_i,p,\A}<\lambda$ and $\beta_{r_i,p,\A}<\lambda$, hence
\[
\frac{|r_i\cap p|}{|r_i|}-\rho_{r_{i-1},p,\A}\alpha_{r_i,p,\A}-\sigma_{r_{i+1},p,\A}\beta_{r_i,p,\A}>\frac{1}{2}-2\lambda-\lambda-\lambda\ge \nu_1.
\]
This implies $\psi(\tau_{r_i,p,\A})=1$. Similarly, we can see $\psi(\tau_{r_i,q,\A})=1$. Thus, $\xi[p,\psi,\A](r_i)=|r_i|=\xi[q,\psi,\A](r_i)$.
\end{proof}

\begin{prop}\label{prop:xi array}
Suppose $g,h\in G$, $x\in X\cup X^{-1}$, and $\C_{g,h,\mu}\cup\C_{gx,h,\mu}=\{r_1<\cdots<r_n\}$. The following hold.
\begin{itemize}
    \item[(1)]
    $|\xi[g,h,\psi](r_i)-\xi[gx,h,\psi](r_i)|<\frac{K^{i-1}}{\nu_1-\nu_0}$ for any $i\in\{1,\cdots,n\}$.
    \item[(2)]
    $\|\xi[g,h,\psi]-\xi[gx,h,\psi]\|_1<\frac{1}{(1-K)(\nu_1-\nu_0)}$.
\end{itemize}
\end{prop}

\begin{proof}
(1) For any $p\in\G_{g,h}$, by Lemma \ref{lem:xi estimate}, there exists $q\in \G_{gx,h}$ such that for any $i\in\{1,\cdots,n\}$,
\[
|\xi[p,\psi,\C_{g,h,\mu}\cup\C_{gx,h,\mu}](r_i)-\xi[q,\psi,\C_{g,h,\mu}\cup\C_{gx,h,\mu}](r_i)|<\frac{K^{i-1}}{\nu_1-\nu_0}.
\]
By Lemma \ref{lem:xi stable}, this implies
\begin{align*}
\xi[p,\psi,\C_{g,h,\mu}](r_i)-\frac{K^{i-1}}{\nu_1-\nu_0}
&=\xi[p,\psi,\C_{g,h,\mu}\cup\C_{gx,h,\mu}](r_i)-\frac{K^{i-1}}{\nu_1-\nu_0}\\
&<\xi[q,\psi,\C_{g,h,\mu}\cup\C_{gx,h,\mu}](r_i)
=\xi[q,\psi,\C_{gx,h,\mu}](r_i)\\
&\le\max_{q'\in\G_{gx,h}}\xi[q',\psi,\C_{gx,h,\mu}](r_i)
= \xi[gx,h,\psi](r_i).
\end{align*}
Hence, for any $i\in\{1,\cdots,n\}$, we have 
\[
\xi[g,h,\psi](r_i)-\frac{K^{i-1}}{\nu_1-\nu_0}
=\max_{p\in\G_{g,h}}\xi[p,\psi,\C_{g,h,\mu}](r_i)-\frac{K^{i-1}}{\nu_1-\nu_0}
<\xi[gx,h,\psi](r_i).
\]
By the same argument for $gx,h\in G$ and $x^{-1}\in X\cup X^{-1}$, we also have 
\[
\xi[gx,h,\psi](r_i)-\frac{K^{i-1}}{\nu_1-\nu_0}<\xi[g,h,\psi](r_i)
\]
for any $i\in\{1,\cdots,n\}$, thus $|\xi[g,h,\psi](r_i)-\xi[gx,h,\psi](r_i)|<\frac{K^{i-1}}{\nu_1-\nu_0}$.

(2) Note $\xi[g,h,\psi](r_i)=\xi[gx,h,\psi](r_i)=0$ for any $r\in\C \sm (\C_{g,h,\mu}\cup\C_{gx,h,\mu})$. Hence, by (1), we have
\begin{align*}
    \|\xi[g,h,\psi]-\xi[gx,h,\psi]\|_1
&=\sum_{i=1}^n |\xi[g,h,\psi](r_i)-\xi[gx,h,\psi](r_i)| \\
&<\sum_{i=1}^n \frac{K^{i-1}}{\nu_1-\nu_0}
<\sum_{i=1}^\infty \frac{K^{i-1}}{\nu_1-\nu_0}
=\frac{1}{(1-K)(\nu_1-\nu_0)}.
\end{align*}
\end{proof}

\subsection{Second array}\label{subsec:second array}
The main goal of Section \ref{subsec:second array} is to define a map on $\E$ in Definition \ref{def:eta} and prove Proposition \ref{prop:eta aray}. Recall Definition \ref{def:E,V} and Remark \ref{rem:E} for $\E$ and $\E(p)$.

\begin{defn}
For $g,h\in G$, we define
\[
\E_{g,h}=\bigcup_{p\in \G_{g,h}}\E(p).
\]
\end{defn}

\begin{defn}\label{def:eta}
For $g,h\in G$, the function $\psi$ in \eqref{eq:psi}, and a subset $\A$ of $\C_{g,h,2\lambda}$ (cf. \eqref{eq:all intersection}), we define a map $\eta[g,h,\psi,\A] \colon \E \to \RR$ by
\[
\eta[g,h,\psi,\A]=\sum_{e\in\E_{g,h}} \left( 1-\max\left\{\frac{\xi[g,h,\psi](r)}{|r|} \mid r\in \A \wedge e\in \E(r) \right\}\right) \cdot 1_e.
\]
In particular, we define
\[
\eta[g,h,\psi]=\eta[g,h,\psi,\C_{g,h,\mu}].
\]
\end{defn}

\begin{rem}\label{rem:eta positive}
For any $e\in \E$, we have $0\le \eta[g,h,\psi](e)\le 1$.
\end{rem}

The following are straightforward.

\begin{lem}
\label{lem:eta sym eqiv}
\begin{itemize}
    \item[(1)]
    $\eta[g,h,\psi]=\eta[h,g,\psi]$ for any $g,h\in G$.
    \item[(2)]
    $\eta[kg,kh,\psi]=\pi_\E(k)\eta[h,g,\psi]$ for any $g,h,k\in G$.
\end{itemize}
\end{lem}

\begin{proof}
(1) follows from $\E_{g,h}=\E_{h,g}$, $\C_{g,h,\mu}=\C_{h,g,\mu}$ and, $\xi[g,h,\psi]=\xi[h,g,\psi]$. (2) follows from $\E_{kg,kh}=k\E_{g,h}$, $\C_{kg,kh,\mu}=k\C_{g,h,\mu}$ and, $\xi[kg,kh,\psi]=\pi_\C(k)\xi[g,h,\psi]$.
\end{proof}

\begin{lem}
\label{lem:eta stable}
For any $g,h\in G$ and $x\in X\cup X^{-1}$, we have
\[
\eta[g,h,\psi]=\eta[g,h,\psi,\C_{g,h,\mu}\cup \C_{gx,h,\mu}].
\]
\end{lem}

\begin{proof}
For any $r\in\C \sm \C_{g,h,\mu}$, we have $\xi[g,h,\psi](r)=0$ by Definition \ref{def:xi[p,psi,A]} and Definition \ref{def:xi[g,h,psi]}. Hence, for any $e\in\E_{g,h}$, we have 
\begin{align*}
    &\max\bigg\{\frac{\xi[g,h,\psi](r)}{|r|} \;\bigg|\; r\in \C_{g,h,\mu} \wedge e\in \E(r) \bigg\}\\
=&\max\bigg\{\frac{\xi[g,h,\psi](r)}{|r|} \;\bigg|\; r\in \C_{g,h,\mu}\cup\C_{gx,h,\mu} \wedge e\in \E(r) \bigg\}.
\end{align*}
This implies $\eta[g,h,\psi]=\eta[g,h,\psi,\C_{g,h,\mu}]=\eta[g,h,\psi,\C_{g,h,\mu}\cup \C_{gx,h,\mu}]$.
\end{proof}

\begin{lem}
\label{lem:eta estimate}
Suppose $g,h\in G$, $x\in X\cup X^{-1}$, and $\C_{g,h,\mu}\cup \C_{gx,h,\mu}=\{r_1 < \cdots < r_n\}$. For any $e\in \E_{g,h}\cap \E_{gx,h}\cap \left(\bigcup_{i=1}^n \E(r_i)\right)$, there exists $i_e\in\{1,\cdots,n\}$ such that $e\in \E(r_{i_e})$ and
\begin{equation}
\label{eq:eta estimate}
    \left|\eta[g,h,\psi,\C_{g,h,\mu}\cup \C_{gx,h,\mu}](e)-\eta[gx,h,\psi,\C_{g,h,\mu}\cup \C_{gx,h,\mu}](e) \right|
\le \frac{K^{i_e-2}}{(\nu_1-\nu_0)|r_{i_e}|}.
\end{equation}
\end{lem}

\begin{proof}
Let $e\in \E_{g,h}\cap\E_{gx,h}\cap \left(\bigcup_{i=1}^n \E(r_i)\right)$. For simplicity, we denote $\A=\C_{g,h,\mu}\cup \C_{gx,h,\mu}$ and $\A_e=\{r\in \C_{g,h,\mu}\cup \C_{gx,h,\mu} \mid e\in \E(r)\}$. Since $e\in\E_{g,h}$, there exists $p\in \G_{g,h}$ such that $e\in\E(p)$. Note $\A\subset \C_{p,2\lambda}$ by Corollary \ref{cor:attached contours}. Hence, by Lemma \ref{lem:three contours}, one of (1) or (2) below occurs: (1) $\A_e=\{r_k\}$ for some $k\in\{1\cdots,n\}$, (2) $\A_e=\{r_k,r_{k+1}\}$ for some $k\in\{1\cdots,n-1\}$. In case (1), we have $\eta[g,h,\psi,\A](e)=1-\frac{\xi[g,h,\psi](r_k)}{|r_k|}$ and $\eta[gx,h,\psi,\A](e)=1-\frac{\xi[gx,h,\psi](r_k)}{|r_k|}$. hence by Proposition \ref{prop:xi array} (1), 
\[
|\eta[g,h,\psi,\A](e)-\eta[gx,h,\psi,\A](e)|
=\frac{|\xi[g,h,\psi](r_k)-\xi[gx,h,\psi](r_k)|}{|r_k|}
<\frac{K^{k-1}}{(\nu_1-\nu_0)|r_k|}.
\]
Thus, $i_e=k$ satisfies \eqref{eq:eta estimate}. Here, note $0<K<1$.

In case (2), if $\frac{\xi[g,h,\psi](r_k)}{|r_k|}\ge \frac{\xi[g,h,\psi](r_{k+1})}{|r_{k+1}|}$, then $\eta[g,h,\psi,\A](e)=1-\frac{\xi[g,h,\psi](r_k)}{|r_k|}$. Hence, by Proposition \ref{prop:xi array} (1), we have
\begin{align*}
    \eta[gx,h,\psi,\A](e)
&\le 1-\frac{\xi[gx,h,\psi](r_k)}{|r_k|}\\
&<1-\frac{\xi[g,h,\psi](r_k)}{|r_k|}+\frac{K^{k-1}}{(\nu_1-\nu_0)|r_k|}\\
&=\eta[g,h,\psi,\A](e)+\frac{K^{k-1}}{(\nu_1-\nu_0)|r_k|}\\
&\le \eta[g,h,\psi,\A](e)+\frac{K^{k-1}}{(\nu_1-\nu_0)\cdot \min\{|r_k|,|r_{k+1}|\}}.
\end{align*}
If $\frac{\xi[g,h,\psi](r_k)}{|r_k|}\le \frac{\xi[g,h,\psi](r_{k+1})}{|r_{k+1}|}$, then we get by similar calculation using Proposition \ref{prop:xi array} (1),
\begin{align*}
    \eta[gx,h,\psi,\A](e)
&< \eta[g,h,\psi,\A](e)+\frac{K^{k}}{(\nu_1-\nu_0)\cdot \min\{|r_k|,|r_{k+1}|\}}\\
&\le \eta[g,h,\psi,\A](e)+\frac{K^{k-1}}{(\nu_1-\nu_0)\cdot \min\{|r_k|,|r_{k+1}|\}}.
\end{align*}
Note $0<K<1$. In both cases, we have $\eta[gx,h,\psi,\A](e)< \eta[g,h,\psi,\A](e)+\frac{K^{k-1}}{(\nu_1-\nu_0)\cdot \min\{|r_k|,|r_{k+1}|\}}$. By the same argument for $gx,h\in G$ and $x^{-1}\in X\cup X^{-1}$, we also have $\eta[g,h,\psi,\A](e)< \eta[gx,h,\psi,\A](e)+\frac{K^{k-1}}{(\nu_1-\nu_0)\cdot \min\{|r_k|,|r_{k+1}|\}}$, hence
\[
|\eta[g,h,\psi,\A](e)- \eta[gx,h,\psi,\A](e)|
<\frac{K^{k-1}}{(\nu_1-\nu_0)\cdot \min\{|r_k|,|r_{k+1}|\}}.
\]
Thus, if $|r_k|\le|r_{k+1}|$, $i_e=k$ satisfies $\eqref{eq:eta estimate}$ and if $|r_k|\ge|r_{k+1}|$, $i_e=k+1$ satisfies $\eqref{eq:eta estimate}$.
\end{proof}

\begin{prop}\label{prop:eta aray}
For any $g,h\in G$ and $x\in X\cup X^{-1}$, we have
\[
\|\eta[g,h,\psi]-\eta[gx,h,\psi]\|_1
<1+\frac{1}{K(1-K)(\nu_1-\nu_0)}.
\]
\end{prop}

\begin{proof}
Let $g,h\in G$, $x\in X\cup X^{-1}$, and $\C_{g,h,\mu}\cup \C_{gx,h,\mu}=\{r_1 < \cdots < r_n\}$. Let $\e$ be the edge of $\Gamma$ from $g$ to $gx$ with $\L(\e)=x$. Note $\eta[g,h,\psi](e)=\eta[gx,h,\psi](e)=0$ for any $e\in \E\sm (\E_{g,h}\cup\E_{gx,h})$. We claim for any $e\in \E_{g,h}\sm(\E(\e)\cup\E_{gx,h})$, $\eta[g,h,\psi](e)=0$. Indeed, since $e\in\E_{g,h}\sm \E_{gx,h}$, there exist $p\in\G_{g,h}$ and $q\in\G_{gx,h}$ such that $e\in \E(p)\sm \E(q)$. Let $\Delta$ be a van Kampen diagram of $\L(\e qp^{-1})$. Since $p\sm q \subset \e \cup \bigcup_{s\in\C_\Delta}(s\cap p)$, we have $e\in \E(\e) \cup \bigcup_{s\in\C_\Delta}\E(s)$. Recall we assume $e\neq \E(\e)$, hence there exists $s_e\in \C_\Delta$ such that $e\in\E(s_e)$. Note that by Lemma \ref{lem:CDelata large} and Corollary \ref{cor:attached contours}, we have $\C_\Delta \subset \C_{p,\frac{1}{2}-2\lambda} \subset \C_{g,h,\frac{1}{2}-2\lambda-2\lambda}$. In particular, we have $\C_\Delta\subset \C_{g,h,\mu}$ by $\mu\le\frac{1}{2}-4\lambda$. We also have $\xi[g,h,\psi](s_e)=|s_e|$ by Lemma \ref{lem:xi for big contour} and $\nu_1+2\lambda\le\frac{1}{2}-2\lambda$. Thus,
\[
0
\le \eta[g,h,\psi](e)
\le 1-\frac{|s_e|}{|s_e|}
=0.
\]
This implies $\eta[g,h,\psi](e)=\eta[gx,h,\psi](e)=0$ for any $e\in \E_{g,h}\sm(\E(\e)\cup\E_{gx,h})$. Similarly we can also see $\eta[g,h,\psi](e)=\eta[gx,h,\psi](e)=0$ for any $e\in \E_{gx,h}\sm(\E(\e)\cup\E_{g,h})$. Thus, we have
\[
\|\eta[g,h,\psi]-\eta[gx,h,\psi]\|_1
=\sum_{e\in \E(\e)\cup(\E_{g,h}\cap\E_{gx,h})}|\eta[g,h,\psi](e)-\eta[gx,h,\psi](e)|.
\]

If $e\in (\E_{g,h}\cap\E_{gx,h})\sm \left(\bigcup_{i=1}^n \E(r_i)\right)$, then we have $\eta[g,h,\psi](e)=\eta[gx,h,\psi](e)=1$. This implies
\begin{align*}
    &\sum_{e\in \E(\e)\cup(\E_{g,h}\cap\E_{gx,h})}|\eta[g,h,\psi](e)-\eta[gx,h,\psi](e)|\\
    =&\sum_{e\in \E(\e)\cup\left(\E_{g,h}\cap\E_{gx,h}\cap\left(\bigcup_{i=1}^n \E(r_i)\right) \right)}|\eta[g,h,\psi](e)-\eta[gx,h,\psi](e)|.
\end{align*}

By Lemma \ref{lem:eta stable} and Lemma \ref{lem:eta estimate}, for any $e\in \E_{g,h}\cap \E_{gx,h}\cap\left(\bigcup_{i=1}^n \E(r_i)\right)$, there exists $i_e\in\{1,\cdots,n\}$ such that $e\in \E(r_{i_e})$ and
\[
\left|\eta[g,h,\psi](e)-\eta[gx,h,\psi](e) \right|
\le \frac{K^{i_e-2}}{(\nu_1-\nu_0)|r_{i_e}|}.
\]
Hence, when we define for each $r_k\in \C_{g,h,\mu}\cup \C_{gx,h,\mu}$,
\begin{align*}
    B_{r_k}=\Big\{e\in& \E_{g,h}\cap\E_{gx,h}\cap \Big(\bigcup_{i=1}^n \E(r_i)\Big) \\
    &\;\Big|\; e\in \E(r_k) \wedge |\eta[g,h,\psi](e)-\eta[gx,h,\psi](e) |
    \le \frac{K^{k-2}}{(\nu_1-\nu_0)|r_k|}\Big\},
\end{align*}
we have $\E_{g,h}\cap\E_{gx,h}\cap \left(\bigcup_{i=1}^n \E(r_i)\right)=\bigcup_{i=1}^n B_{r_i}$. Thus, we have
\begin{align*}
    \|\eta[g,h,\psi]-\eta[gx,h,\psi]\|_1
    &=\sum_{e\in \E(\e)\cup\left(\E_{g,h}\cap\E_{gx,h}\cap\left(\bigcup_{i=1}^n \E(r_i)\right) \right)}|\eta[g,h,\psi](e)-\eta[gx,h,\psi](e)|\\
    &\le |\eta[g,h,\psi](\E(\e))-\eta[gx,h,\psi](\E(\e))|\\
    &~~~~~~~~~~~
    +\sum_{i=1}^n \sum_{e\in B_{r_i}}|\eta[g,h,\psi](e)-\eta[gx,h,\psi](e)|\\
    &\le 1+\sum_{i=1}^n \sum_{e\in B_{r_i}} \frac{K^{i-2}}{(\nu_1-\nu_0)|r_i|}
    =1+\sum_{i=1}^n \frac{K^{i-2}}{(\nu_1-\nu_0)|r_i|}\cdot|B_{r_i}|\\
    &\le 1+\sum_{i=1}^n \frac{K^{i-2}}{\nu_1-\nu_0}\\
    &< 1+\sum_{i=1}^\infty \frac{K^{i-2}}{\nu_1-\nu_0}
    =1+\frac{1}{K(1-K)(\nu_1-\nu_0)}.
\end{align*}
\end{proof}

\begin{lem}
\label{lem:eta proper}
Suppose $g,h\in G$, $p\in \G_{g,h}$, and $e\in\E(p)$. If $e\notin \bigcup_{r\in \C_{p,\nu_0-2\lambda}}\E(r)$, then $\eta[g,h,\psi](e)=1$.
\end{lem}

\begin{proof}
We will show the contraposition. Suppose $e\in\E(p)$ and $\eta[g,h,\psi](e)<1$, then there exists $s\in\C_{g,h,\mu}$ such that $e\in\E(s)$ and $\xi[g,h,\psi](s)>0$. Hence, there exists $q\in\G_{g,h}$ such that $\xi[q,\psi,\C_{g,h,\mu}](s)>0$. This implies $\tau_{s,q,\C_{g,h,\mu}}>\nu_0$, hence we have
\[
\frac{|s\cap q|}{|s|}\ge \tau_{s,q,\C_{g,h,\mu}}>\nu_0.
\]
by Remark \ref{rem:tau}. By Lemma \ref{lem:attached contours}, this implies $s\in\C_{p,\nu_0-2\lambda}$. Thus, we have $e\in \bigcup_{r\in \C_{p,\nu_0-2\lambda}}\E(r)$.
\end{proof}

\section{Main theorem}\label{sec:main thm}
The goal of this section is to prove Proposition \ref{prop:intro proper array} and Theorem \ref{thm:main}. Up to Remark \ref{rem:proper array}, $G,\Gamma,\lambda,\mu$ are as were defined at the beginning of Section \ref{sec:construction of arrays}. That is, $G=\la X \mid \R \ra$ is a finitely generated $C'(\lambda)$-group satisfying the condition \eqref{eq:*} in Section \ref{sec:construction of arrays}, $\Gamma$ is the Cayley graph of $G$ with respect to $X$, $\lambda=\frac{1}{33}$, and $\mu=\frac{4}{33}=4\lambda$. We fix
\[
\nu_{1,0}=\frac{6}{33}=6\lambda  ,~~  \nu_{1,1}=\frac{7.1}{33}=7.1\lambda, ~~\nu_{2,0}=\frac{11.1}{33}=11.1\lambda  ,~~  \nu_{2,1}=\frac{12.2}{33}=12.2\lambda,
\]
and define two functions $\psi_1,\psi_2 \colon \RR\to\RR$ by
\[
    \psi_1(x)=
    \begin{cases}
    0 
    & {\rm if~}x\le \nu_{1,0}\\
    \frac{1}{\nu_{1,1}-\nu_{1,0}}(x-\nu_{1,0}) 
    &{\rm if~}\nu_{1,0}<x<\nu_{1,1} \\
    1 
    & {\rm if~}x\ge \nu_{1,1}
    \end{cases}
\]
and
\[
    \psi_2(x)=
    \begin{cases}
    0 
    & {\rm if~}x\le \nu_{2,0}\\
    \frac{1}{\nu_{2,1}-\nu_{2,0}}(x-\nu_{2,0}) 
    &{\rm if~}\nu_{2,0}<x<\nu_{2,1} \\
    1 
    & {\rm if~}x\ge \nu_{2,1}.
    \end{cases}
\]
\begin{rem}
Note that both of $(\nu_{1,0},\nu_{1,1})$ and $(\nu_{2,0},\nu_{2,1})$ satisfy the inequalities (\ref{eq:1nu}) and (\ref{eq:2nu}) for ($\nu_0,\nu_1$). We also have 
\begin{equation}
\label{eq:nu1,1 nu2,0}
    \nu_{1,1}+2\lambda \le \nu_{2,0}-2\lambda.
\end{equation}
\end{rem}

We denote
\begin{equation}\label{eq:K1K2}
    K_1=\frac{\lambda}{\nu_{1,1}-\nu_{1,0}}
    {\rm ~~and~~}
    K_2=\frac{\lambda}{\nu_{2,1}-\nu_{2,0}}.
\end{equation}

Lemma \ref{lem:to show proper} is for showing properness in the proof of Theorem \ref{thm:main}.

\begin{lem}\label{lem:to show proper}
For any $g,h\in G$, we have
\[
d(g,h)
\le \|\xi[g,h,\psi_1]\|_1+\|\eta[g,h,\psi_2]\|_1.
\]
\end{lem}

\begin{proof}
Fix $p\in\G_{g,h}$. By Lemma \ref{lem:xi for big contour}, we have for any $r\in\C_{p,(\nu_{1,1}+2\lambda)}$, $\xi[g,h,\psi_1](r)=|r|$. This implies
\[
\sum_{r\in \C_{p,(\nu_{1,1}+2\lambda)}}|r|
=\sum_{r\in \C_{p,(\nu_{1,1}+2\lambda)}}\xi[g,h,\psi_1](r)
\le \|\xi[g,h,\psi_1]\|_1.
\]
On the other hand, given $e\in\E(p)\sm \bigcup_{r\in \C_{p,(\nu_{1,1}+2\lambda)}}\E(r)$, we have $e\in\E(p)\sm \bigcup_{r\in \C_{p,(\nu_{2,0}-2\lambda)}}\E(r)$ since \eqref{eq:nu1,1 nu2,0} implies $\C_{p,(\nu_{2,0}-2\lambda)} \subset \C_{p,(\nu_{1,1}+2\lambda)}$. Hence, we have $\eta[g,h,\psi_2](e)=1$ for any $e\in\E(p)\sm \bigcup_{r\in \C_{p,(\nu_{1,1}+2\lambda)}}\E(r)$ by applying Lemma \ref{lem:eta proper} to $\nu_0 = \nu_{2,0}$ and $\psi = \psi_2$. This implies
\[
\Big|p\sm \bigcup_{r\in\C_{p,(\nu_{1,1}+2\lambda)}}r \Big|
=\sum_{e\in\E(p)\sm \bigcup_{r\in \C_{p,(\nu_{1,1}+2\lambda)}}\E(r)}\eta[g,h,\psi_2](e)
\le \|\eta[g,h,\psi_2]\|_1.
\]
Thus, we have
\begin{align*}
    d(g,h)
    =|p|
    &=
    \Big|p\cap \bigcup_{r\in\C_{p,(\nu_{1,1}+2\lambda)}}r \Big|
    +
    \Big|p\sm \bigcup_{r\in\C_{p,(\nu_{1,1}+2\lambda)}}r \Big| \\
    &\le
    \sum_{r\in\C_{p,(\nu_{1,1}+2\lambda)}}|r|
    +
    \Big|p\sm \bigcup_{r\in\C_{p,(\nu_{1,1}+2\lambda)}}r \Big| \\
    &\le
    \|\xi[g,h,\psi_1]\|_1+\|\eta[g,h,\psi_2]\|_1.
\end{align*}
\end{proof}

Next, we will turn maps constructed in Section \ref{subsec:first array} and \ref{subsec:second array} into maps on $G$. Recall Definition \ref{def:E,V} for $\V(r)$ and $\V(e)$.

\begin{defn}
We define a linear map $P_\C \colon \ell^1(\C) \to \ell^1(G)$ by defining for $\xi=\sum_{r\in\C} \xi(r) 1_r \in \ell^1(\C)$,
\[
P_\C \xi
=\sum_{r\in\C} \frac{\xi(r)}{|r|} 1_{\V(r)},
\]
where $1_{\V(r)}(g)=1$ if $g\in \V(r)$ and $1_{\V(r)}(g)=0$ if $g\in G\sm\V(r)$.
\end{defn}

\begin{rem}\label{rem:PC}
We can see $P_\C\xi$ is indeed in $\ell^1(G)$ by the following computation. For each $g\in G$, we define $C_g=\{r\in\C \mid g\in\V(r)\}$, then we have
\begin{align*}
    \|P_\C \xi\|_1
    &= \Big\|\sum_{g\in G} \Big(\sum_{r\in C_g}\frac{\xi(r)}{|r|}\Big) 1_g\Big\|_1 
    = \sum_{g\in G} \Big|\sum_{r\in C_g}\frac{\xi(r)}{|r|}\Big| \\
    &\le \sum_{g\in G} \Big(\sum_{r\in C_g}\frac{|\xi(r)|}{|r|}\Big) \\
    &= \sum_{r\in\C} \Big(\sum_{g\in\V(r)}\frac{|\xi(r)|}{|r|}\Big)
    = \sum_{r\in\C} \frac{|\xi(r)|}{|r|}\cdot|\V(r)| \\
    &= \sum_{r\in\C} |\xi(r)|
    =\|\xi\|_1.
\end{align*}
The above computation also shows that if $\xi(r)\ge 0$ for any $r\in\C$, then $\|P_\C \xi\|_1=\|\xi\|_1$. Indeed, the only inequality above becomes equality in this case.
\end{rem}

\begin{defn}
We define a linear map $P_\E \colon \ell^1(\E) \to \ell^1(G)$ by defining for $\eta=\sum_{e\in\E} \eta(e) 1_e \in \ell^1(\E)$,
\[
P_\E \eta
=\sum_{e\in\E} \frac{\eta(e)}{2} 1_{\V(e)}.
\]
\end{defn}

\begin{rem}\label{rem:PE}
We can see $P_\E \eta$ is indeed in $\ell^1(G)$ in the same way as Remark \ref{rem:PC}. For each $g\in G$, we define $E_g=\{e\in\E \mid g\in\V(e)\}$, then we have
\begin{align*}
    \|P_\E \eta\|_1
    &= \Big\|\sum_{g\in G} \Big(\sum_{e\in E_g}\frac{\eta(e)}{2}\Big) 1_g\Big\|_1 
    = \sum_{g\in G} \Big|\sum_{e\in E_g}\frac{\eta(e)}{2}\Big| \\
    &\le \sum_{g\in G} \Big(\sum_{e\in E_g}\frac{|\eta(e)|}{2}\Big)
    = \sum_{e\in\E} \Big(\sum_{g\in\V(e)}\frac{|\eta(e)|}{2}\Big)
    = \sum_{e\in\E} |\eta(e)|
    =\|\eta\|_1.
\end{align*}
The above computation also shows that if $\eta(e)\ge 0$ for any $e\in\E$, then $\|P_\E \eta\|_1=\|\eta\|_1$. Indeed, the only inequality above becomes equality in this case.
\end{rem}

Now, we prove Proposition \ref{prop:intro proper array}, which is Proposition \ref{prop:proper array} below. Recall that $G=\la X \mid \R \ra$ is a finitely generated $C'(\lambda)$-group with $\lambda=\frac{1}{33}$ satisfying condition \eqref{eq:*}. For $p>0$, $\lambda_G \colon G\act\ell^p(G)$ denotes the left regular representation of $G$ as usual.

\begin{prop}\label{prop:proper array}
There exists a map $\Phi \colon G\times G\ni (g,h) \mapsto \Phi[g,h]\in\ell^1(G)$ satisfying the following.
\begin{itemize}
    \item[(1)]
    For any $g,h\in G$, $\Phi[g,h]$ has finite support and $\Phi[g,h](k)\ge 0$ for any $k\in G$.
    \item[(2)]
    $\Phi$ is symmetric, i.e. $\Phi[g,h]=\Phi[h,g]$ for any $g,h\in G$.
    \item[(3)]
    $\Phi$ is $G$-equivariant, i.e. $\Phi[kg,kh]=\lambda_G(k)\Phi[h,g]$ for any $g,h,k\in G$.
    \item[(4)]
    $d(g,h)\le \|\Phi[g,h]\|_1$ for any $g,h\in G$.
    \item[(5)]
    There exists $L\in \RR_{>0}$ such that $\|\Phi[g,h]-\Phi[gk,h]\|_1\le L\cdot d(1,k)$ for any $g,h,k\in G$.
\end{itemize}
\end{prop}

\begin{proof}
We define a map $\Phi \colon G\times G \to \ell^1(G)$ by
\[
\Phi[g,h]=P_\C \xi[g,h,\psi_1]+P_\E \eta[g,h,\psi_2].
\]
It's straightforward to see $\Phi$ is symmetric and $G$-equivariant by Lemma \ref{lem:xi sym eqiv} and Lemma \ref{lem:eta sym eqiv}. Since $\xi[g,h,\psi_1]$ and $\eta[g,h,\psi_2]$ both have finite support, so does $\Phi[g,h]$. We also have $\Phi[g,h](k)\ge 0$ for any $k\in G$ by Remark \ref{rem:xi positive} and Remark \ref{rem:eta positive}. This also implies
\[
\|\Phi[g,h]\|_1
= \|P_\C \xi[g,h,\psi_1]\|_1+\|P_\E \eta[g,h,\psi_2]\|_1
= \|\xi[g,h,\psi_1]\|_1+\|\eta[g,h,\psi_2]\|_1
\ge d(g,h),
\]
using Remark \ref{rem:PC}, Remark \ref{rem:PE}, and Lemma \ref{lem:to show proper}.

It remains to show condition (5). Let $g,h,k\in G$ and $k=x_1\cdots x_n$, where $n=d(1,k)$ and $x_i\in X\cup X^{-1}$ for each $i\in \{1,\cdots,n\}$. We denote $x_0=1$ for convenience. By Proposition \ref{prop:xi array} (2) and Proposition \ref{prop:eta aray} (cf. \eqref{eq:K1K2}), we have
\begin{align*}
    \|\Phi[g,h]-\Phi[gk,h]\|_1
    &\le \sum_{i=1}^n\|\Phi[gx_0\cdots x_{i-1},h]-\Phi[gx_0\cdots x_i,h]\|_1 \\
    &\le \sum_{i=1}^n \Big(\|P_\C \xi[gx_0\cdots x_{i-1},h,\psi_1]-P_\C \xi[gx_0\cdots x_i,h,\psi_1]\|_1 \\
    &~~~~~~~~~~~~~~+\|P_\E \eta[gx_0\cdots x_{i-1},h,\psi_2]-P_\E \eta[gx_0\cdots x_i,h,\psi_2]\|_1 \Big)\\
    &= \sum_{i=1}^n \Big(\|P_\C (\xi[gx_0\cdots x_{i-1},h,\psi_1]-\xi[gx_0\cdots x_i,h,\psi_1])\|_1 \\
    &~~~~~~~~~~~~~~+\|P_\E (\eta[gx_0\cdots x_{i-1},h,\psi_2]-\eta[gx_0\cdots x_i,h,\psi_2])\|_1 \Big)\\
    &\le \sum_{i=1}^n \Big(\|\xi[gx_0\cdots x_{i-1},h,\psi_1]-\xi[gx_0\cdots x_i,h,\psi_1]\|_1 \\
    &~~~~~~~~~~~~~~+\|\eta[gx_0\cdots x_{i-1},h,\psi_2]-\eta[gx_0\cdots x_i,h,\psi_2]\|_1 \Big)\\
    &\le \sum_{i=1}^n \Big( \frac{1}{(1-K_1)(\nu_{1,1}-\nu_{1,0})}
    + 1+\frac{1}{K_2(1-K_2)(\nu_{2,1}-\nu_{2,0})} \Big)\\
    &=n\cdot \Big( \frac{1}{(1-K_1)(\nu_{1,1}-\nu_{1,0})}
    + 1+\frac{1}{K_2(1-K_2)(\nu_{2,1}-\nu_{2,0})} \Big).
\end{align*}
Since $n=d(1,k)$, (5) is satisfied with $L=\frac{1}{(1-K_1)(\nu_{1,1}-\nu_{1,0})}
    + 1+\frac{1}{K_2(1-K_2)(\nu_{2,1}-\nu_{2,0})}$.
\end{proof}

\begin{rem}\label{rem:proper array}
For the map $\Phi$ in Proposition \ref{prop:proper array}, we define a map $\tPhi \colon G\times G\ni (g,h) \mapsto \tPhi[g,h]\in\ell^2(G)$ by \[
\tPhi[g,h](k)=\sqrt{\Phi[g,h](k)}
\]
for each $g,h,k\in G$. Note that $\tPhi$ is also symmetric and $G$-equivariant (cf. Proposition \ref{prop:proper array} (2)(3)). Proposition \ref{prop:proper array} (4) implies for any $g,h\in G$,
\begin{equation}\label{eq:tPhi proper}
    d(g,h)\le \|\Phi[g,h]\|_1=\|\tPhi[g,h]\|_2^2.
\end{equation}
Also, since $|\sqrt{a}-\sqrt{b}|^2\le|\sqrt{a}-\sqrt{b}|(\sqrt{a}+\sqrt{b})=|a-b|$ for any $a,b\in\RR_{\ge0}$, Proposition \ref{prop:proper array} (5) implies for any $g,h,k\in G$,
\begin{equation}\label{eq:tPhi array}
    \|\tPhi[g,h]-\tPhi[gk,h]\|_2^2\le \|\Phi[g,h]-\Phi[gk,h]\|_1\le L\cdot d(1,k).
\end{equation}
\end{rem}

\begin{proof}[Proof of Theorem \ref{thm:main}]
As discussed in Remark \ref{rem:relation}, it's enough to show any finitely generated $C'(\lambda)$-group $G=\la X\mid \R\ra$ with $\lambda=\frac{1}{33}$ satisfying condition \eqref{eq:*} is bi-exact. We will verify Proposition \ref{prop:bi-exact} (3). By \cite[Theorem 1]{Sl}, $G$ has finite asymptotic dimension, hence $G$ is exact.

We define a map $c\colon G \to \ell^2(G)$ by
\[
c(g)=\tPhi[1,g],
\]
where $\tPhi$ is a map in Remark \ref{rem:proper array} defined from a map $\Phi$ of Proposition \ref{prop:proper array}. Note that $(\ell^2(G),\lambda_G)$ is obviously weakly contained in $(\ell^2(G),\lambda_G)$. By \eqref{eq:tPhi proper}, we have for any $n\in \NN$,
\[
\{g\in G \mid \|c(g)\|_2\le n\}\subset
\{g\in G \mid d(1,g)\le n^2\}.
\]
Since $X$ is finite, $\Gamma$ is locally finite, hence the right hand side above is finite. Hence, $c$ is proper.

Finally, since $\tPhi$ is symmetric and $G$-equivariant and satisfies \eqref{eq:tPhi array}, we have for any $g,h,k\in G$,
\begin{align*}
    \|c(gkh)-\lambda_G(g)c(k)\|_2
    &=\|\tPhi[1,gkh]-\lambda_G(g) \tPhi[1,k]\|_2 \\
    &\le\|\tPhi[1,gkh]-\tPhi[1,gk]\|_2+\|\tPhi[1,gk]-\lambda_G(g)\tPhi[1,k]\|_2 \\
    &=\|\tPhi[gkh,1]-\tPhi[gk,1]\|_2+\|\tPhi[1,gk]-\tPhi[g,gk]\|_2 \\
    &\le \sqrt{Ld(1,h)}+\sqrt{Ld(1,g)}.
\end{align*}
This implies $\sup_{k\in G}\|c(gkh)-\lambda_G(g)c(k)\|_2<\infty$ for any $g,h\in G$.
\end{proof}

\begin{rem}\label{rem:the end of proof of Theorem 1.1}
    By the same argument as the end of the proof of Theorem \ref{thm:main}, the following holds. Let $G$ be a group, $\mathcal{K}$ a Hilbert space, and $\pi\colon G\to \mathcal{U}(\mathcal{K})$ a unitary representation. We denote $\pi(g) \in \mathcal{U}(\mathcal{K})$ by $\pi_g$ for each $g \in G$. Suppose that a map $r \colon G \times G \to \mathcal{K}$ satisfies (1), (2), and (3) below.
    \begin{itemize}
    \item[(1)]
    $\big[ \forall\, g,h \in G,\,r(g,h) = r(h,g) \big] \vee \big[\forall\, g,h \in G,\,r(g,h) = -r(h,g)\big]$.
    \item[(2)]
    $\forall\,g,h,k \in G,\, \pi_k(r(g,h)) = r(kg,kh)$.
    \item[(3)]
    $\forall\, g \in G,\, \sup_{k \in G} \|r(1,k) - r(g,k)\| < \infty$.
    \end{itemize}
    Then, the map $c \colon G \to \mathcal{K}$ defined by $c(g) = r(1,g)$ satisfies $\sup_{k\in G}\|c(gkh)-\pi_g (c(k))\|<\infty$ for all $g,h\in G$.
\end{rem}

\section{Generalization to the infinitely generated case} \label{sec:infinite}

In this section, we generalize Theorem \ref{thm:main} to some infinitely generated small cancellation groups. Note that in Theorem \ref{thm:infinite case}, both $X$ and $\R$ can be infinite.

\begin{lem} \label{lem:free product}
    If $\{G_n\}_{n\in\NN}$ are countable bi-exact groups, then $\ast_{n\in\NN}G_n$ is bi-exact.
\end{lem}

\begin{proof}
    Since a free product of two exact groups is exact and an increasing union of exact groups are exact (cf. \cite{BO}), $\ast_{n=1}^N G_n$ is exact for any $N\in\NN$ and this implies that $\ast_{n\in\NN}G_n$ is exact. For each $n\in\NN$, there exists a proper array $r_n$ of $G_n$ into the left regular representation $(\ell^2(G_n),\lambda_{G_n})$ by Proposition \ref{prop:bi-exact} (2). We claim that for any $n\in\NN$, we can assume
    \begin{align}\label{eq:r_n}
        \|r_n(g)\|_2\ge n
    \end{align}
    for any $g\in G_n \setminus \{1\}$. Indeed, since $r_n$ is proper, the set $F_n=\{g\in G_n \mid \|r_n(g)\|_2 < n \}$ is finite. Note $F_n=F_n^{-1}$ by Definition \ref{def:array} (1). Define a map $r'_n \colon G_n \to \ell^2(G_n)$ by
    \[ r'_n(g)=
        \begin{cases}
    n(1_g-1_1) & \mathrm{if}\;\; g\in F_n \\
    r_n(g) & \mathrm{if}\;\; g\in G_n\setminus F_n,
    \end{cases}
    \]
    where $1_g\in \ell^2(G_n)$ is defined by $1_g(g)=1$ and $1_g(x)=0$ if $x\neq g$. Note $\|r'_n(g)\|_2 =\sqrt{2}n$ for any $g \in F_n\setminus\{1\}$. This implies $\|r'_n(g)\|_2\ge n$ for any $g\in G_n \setminus \{1\}$. It's straightforward to see $r'_n$ satisfies Definition \ref{def:array} (1). Since $F_n$ is finite, we have $\sup_{g\in G_n} \|r_n(g)-r'_n(g)\|_2 < \infty$. Hence, $r'_n$ is proper and satisfies Definition \ref{def:array} (2) as well. Thus, the claim follows by replacing $r_n$ by $r'_n$.

    We denote $G=\ast_{n\in\NN}G_n$ for simplicity and regard each $r_n$ as a map from $G_n$ to $\ell^2(G)$ by composing it with the embedding $\ell^2(G_n) \inj \ell^2(G)$. Consider the unitary representation $(\lambda_G\otimes {\rm id}_{\ell^2(\NN)},\ell^2(G)\otimes \ell^2(\NN))$ of $G$ and let $\{e_n\}_{n\in\NN}$ be an orthonormal basis of $\ell^2(\NN)$ defined by $e_n=1_{\{n\}} \in\ell^2(\NN)$ for each $n\in\NN$. Since $(\lambda_G\otimes {\rm id}_{\ell^2(\NN)},\ell^2(G)\otimes \ell^2(\NN))$ is unitarily isomorphic to a direct sum representation $\bigoplus_{n\in\NN} (\lambda_G, \ell^2(G))$, $(\lambda_G\otimes {\rm id}_{\ell^2(\NN)},\ell^2(G)\otimes \ell^2(\NN))$ is weakly contained in $(\lambda_G, \ell^2(G))$. For $g\in G\setminus\{1\}$, let $g=h_1\cdots h_m$ be the normal form of $g$, where $m\in\NN$ and $h_i\in G_{n_i}\setminus\{1\}$ with $n_i\neq n_{i+1}$ for any $i$. For convenience, we denote $h_0=1$ and define a map $R\colon G\to \ell^2(G)\otimes \ell^2(\NN)$ by $R(1)=0$ and if $g\neq 1$,
    \[
    R(g)=\sum_{i=1}^m \lambda_G(h_0\cdots h_{i-1})r_{n_i}(h_i)\otimes e_{n_i}.
    \]
    We claim $R$ is proper. Indeed, let $N\in\NN$ and $g\in G\setminus \{1\}$ satisfy $\|R(g)\| \le N$. Let $g=h_1\cdots h_m$ be the normal form of $g$, then it's straightforward to see
    \begin{align}\label{eq:R(g)}
            \|R(g)\|^2 =\sum_{i=1}^m \|r_{n_i}(h_i)\|_2^2 \le N^2.
    \end{align}
    For the equality in \eqref{eq:R(g)}, we used that supports of $\lambda_G(h_0\cdots h_{i-1})r_{n_i}(h_i)$ and $\lambda_G(h_0\cdots h_{j-1})r_{n_j}(h_j)$ are disjoint when $i\neq j$ and $n_i=n_j$. Since $h_i\neq 1$ for each $i$, \eqref{eq:r_n} and \eqref{eq:R(g)} imply $m \le N^2$, and $n_i \le N$ and $\|r_{n_i}(h_i)\|_2 \le N$ for any $i\in\{1,\cdots,m\}$. Note that the set $\{x\in G_{n_i} \mid \|r_{n_i}(x)\|_2 \le N\}$ is finite, since $r_{n_i}$ is proper. This implies
    \[
    \#\{g\in G \mid \|R(g)\| \le N\}
    \le
    \Big(\sum_{n=1}^N \#\{x\in G_n \mid \|r_n(x)\|_2 \le N\} \Big)^{N^2}
    < \infty.
    \]
    Hence, $R$ is proper. It's not difficult to see $R$ satisfies Definition \ref{def:array} (1) and for any $g\in G\setminus\{1\}$ with normal form $g=h_1\cdots h_m$,
    \[
    \sup_{k\in G}\|R(gk)-\lambda_G(g)R(k)\|
    \le \sum_{i=1}^m \sup_{x\in G_{n_i}}\|r_{n_i}(h_ix)-\lambda_{G_{n_i}}(h_i)r_{n_i}(x)\|_2
    <\infty.
    \]
    From this, we can verify Proposition \ref{prop:bi-exact} (3) in the same way as the proof of Theorem \ref{thm:main} (see Remark \ref{rem:the end of proof of Theorem 1.1}).
\end{proof}

\begin{thm} \label{thm:infinite case}
    Suppose that $G = \la X \mid \R \ra$ is $C'(\lambda)$-group with $\lambda < \frac{1}{33}$ and $X$ is countable. If there exists $N \in\NN$ such that for any letter $x\in X$,
    \[
    \#\{y\in X \mid \exists r \in \R {\rm ~s.t.~} x,y\in r \} \le N,
    \] then $G$ is bi-exact.
\end{thm}

\begin{proof}
    We define a graph $\Lambda$ as follows. The vertex set $V(\Lambda)$ is $X$ and two distinct vertices $x,y \in X$ are connected by an edge if and only if there exists $r\in \R$ such that $x,y \in\R$. Let $N \in\NN$ satisfy for any $x\in X$, $\#\{y\in X \mid \exists r \in \R {\rm ~s.t.~} x,y\in r \} \le N$. This implies that $\Lambda$ is uniformly locally finite, that is, for any vertex $x \in V(\Lambda)$, the valency of $x$ is at most $N$. Let $\{\Lambda_i\}_{i\in I}$ be connected components of $\Lambda$, that is, each $\Lambda_i$ is a maximal connected subgraph of $\Lambda$ and $V(\Lambda) = \bigsqcup_{i\in I}V(\Lambda_i)$. Since $X$ is countable, $I$ is countable. For each $i\in I$, define $\R_i = \{r \in \R \mid \exists x \in V(\Lambda_i) {\rm ~s.t.~} x \in r\}$. By maximality of each $\Lambda_i$, we can see $\R=\bigsqcup_{i\in I}\R_i$, hence
    \[
    G \cong \ast_{i\in I}\la V(\Lambda_i) \mid \R_i \ra.
    \]
    If all $\la V(\Lambda_i) \mid \R_i \ra$'s are bi-exact, then $G$ is bi-exact by Lemma \ref{lem:free product}. Hence, we assume $\Lambda$ is connected in the following and it's enough for the proof.
    
    We will embed $G$ into a finitely generated $C'(\frac{1}{33})$-group. Since $\lambda < \frac{1}{33}$, there exists $M \in\NN$ such that
    \begin{align}\label{eq:M}
           \Big(1-\frac{2}{M}\Big)^{-1}\Big[\Big(1+\frac{1}{M}\Big)\lambda + \frac{2}{M} \Big] < \frac{1}{33}.
    \end{align}
    Let $F=\la a_1,\cdots a_{N+1},b,c \ra$ be a free group of rank $N+3$. Note that $\la w c w^{-1} \mid w\in \la a_1,\cdots,a_{N+1}, b \ra \ra$ is a free subgroup of rank $\infty$ with basis $\{w c w^{-1} \mid w\in \la a_1,\cdots,a_{N+1}, b \ra \}$. Define
    \[
    A_n = \{ (wb)^M c (wb)^{-M} \mid w\in \la a_1,\cdots,a_{N+1} \ra, |w|=n \}
    \]
    for each $n \in \NN \cup \{0\}$ and $A = \bigcup_{n \in\NN \cup \{0\}} A_n$. We define an injective map $\psi \colon X \to A$ as follows. Fix a vertex $o \in V(\Lambda)$ and define $S(n,o)=\{x\in V(\Lambda) \mid d_{\Lambda} (o,x)=n \}$ for each $n\in\NN\cup \{0\}$. Since $\# S(n,o) \le N^n$ and $\# A_n \ge N^n$ for any $n\in\NN \cup \{0\}$, we can take an injective map $\psi_n \colon S(n,o) \to A_{n+M}$ for each $n \in \NN \cup \{0\}$. We define $\psi = \bigcup_{n\in \NN \cup \{0\}} \psi_n$. Note that for any relation $r \in \R$ and any letters $x,y \in r$, $x$ and $y$ are connected by an edge of $\Lambda$. This implies $|d_{\Lambda}(o,x)-d_{\Lambda}(o,y)|\le d_{\Lambda}(x,y)\le 1$. Hence, for any $r \in\R$, there exists $m\in \NN$ with $m\ge M$ such that
    \begin{align} \label{eq:length of psi}
         \{|\psi(x)| \mid x \in r \} \subset \{2(m+1)M+1 , 2(m+2)M+1\}.
    \end{align}
    In exactly the same way, we can define another free group $F'=\la a'_1,\cdots a'_{N+1},b',c' \ra$, $A'_n = \{ (w'b')^M c' (w'b')^{-M} \mid w'\in \la a'_1,\cdots,a'_{N+1} \ra, |w'|=n \}$ ($n\in\NN\cup\{0\}$), $A' = \bigcup_{n \in\NN \cup \{0\}} A'_n$, and an injective map $\psi' \colon X \to A'$. That is, for any $x \in X$, $\psi'(x)$ is a word obtained by replacing $a_1,\cdots,a_{N+1},b,c$ in $\psi(x)$ by $a'_1,\cdots,a'_{N+1},b',c'$ respectively. Define a group $H$ by
    \[
    H = \la (G*F)*F' \mid x\psi(x)=\psi'(x), x\in X \ra.
    \]
    Note that $\{\psi(x) \mid x \in X\}$ and $\{\psi'(x) \mid x \in X\}$ generate a free subgroup of $F$ and $F'$ respectively. Also, $\{x\psi(x) \mid x \in X\}$ generates a free subgroup of $G*F$, since the map $q \colon G*F \to F$ defined by $q|_G \equiv 1$ and $q|_F={\rm id}_F$ satisfies $q(x\psi(x))=\psi(x)$ for any $x\in X$. Hence, $H$ is an amalgamated free product of $G*F$ and $F'$ with isomorphic subgroups $\la x\psi(x) \mid x \in X\ra$ and $\la\psi'(x) \mid x \in X\ra$. Thus, $G$ is a subgroup of $H$.
    
    We will show that $H$ is a finitely generated $C'(\frac{1}{33})$-group. For a word $w$ of $X$, we denote by $w(\{\psi'(x)\psi(x)^{-1}\}_{x\in X})$ the word of $\{a_1,\cdots a_{N+1},b,c,a'_1,\cdots a'_{N+1},b',c'\}$ obtained by replacing each $x \in X$ in $w$ by $\psi'(x)\psi(x)^{-1}$. Note that $w(\{\psi'(x)\psi(x)^{-1}\}_{x\in X})$ is not necessarily reduced, even if $w$ is reduced. We define $\tR$ to be the set of all cyclically reduced words of $\{a_1,\cdots a_{N+1},b,c,a'_1,\cdots a'_{N+1},b',c'\}$ one of whose cyclic permutations is obtained by cyclically reducing $r(\{\psi'(x)\psi(x)^{-1}\}_{x\in X})$ with some $r \in \R$. Since $\R$ is symmetric, $\tR$ is also symmetric. By Tietze transformation, we can see
    \begin{align*}
        H 
        &\cong \la a_1,\cdots a_{N+1},b,c,a'_1,\cdots a'_{N+1},b',c' \mid r(\{\psi'(x)\psi(x)^{-1}\}_{x\in X}), r\in \R \ra \\
        &= \la a_1,\cdots a_{N+1},b,c,a'_1,\cdots a'_{N+1},b',c' \mid  \tR \ra.
    \end{align*}
    Note $\{\psi'(x)\psi(x)^{-1} \mid x \in X\}$ satisfies $C'(\frac{1}{M})$-condition. Let $u,v \in\tR$ be distinct words and $w$ be the piece for $u,v \in  \R$. There exist words $r,s \in \R$ such that one of cyclic permutations of $u,v$ is obtained by cyclically reducing $r(\{\psi'(x)\psi(x)^{-1}\}_{x\in X}), s(\{\psi'(x)\psi(x)^{-1}\}_{x\in X})$ respectively. Let $t$ be the piece for $r$ and $s$. Since $r,s$ can also be their cyclic permutations, we assume without loss of generality that $|t|$ is the biggest among all the length of pieces for some distinct cyclic permutations of $r$ and $s$. For $r$, let $m\in\NN$ with $m\ge M$ satisfy \eqref{eq:length of psi}. Using $\{\psi'(x)\psi(x)^{-1} \mid x \in X\}$ satisfies $C'(\frac{1}{M})$-condition, it's not difficult to see
    \begin{align}\label{eq:|w|}
        |w| 
        < (2(m+2)M+1)|t| + \frac{2}{M}|r|
        < (2(m+2)M+1)\lambda|r| + \frac{2}{M}|r|
    \end{align}
    and
    \begin{align}\label{eq:|u|}
        |u|> \Big(1-\frac{2}{M}\Big)(2(m+1)M+1)|r|.
    \end{align}
    The term $\frac{2}{M}|r|$ in \eqref{eq:|w|} comes from possible common letters next to both ends of $t(\{\psi'(x)\psi(x)^{-1}\}_{x\in X})$. \eqref{eq:|w|} and \eqref{eq:|u|} imply
    \begin{align*}
        |w|
        &< \Big((2(m+2)M+1)\lambda + \frac{2}{M}\Big)|r| \\
        &< \frac{(2(m+2)M+1)\lambda + \frac{2}{M}}{\big(1-\frac{2}{M}\big)(2(m+1)M+1)}|u| \\
        &=\Big(1-\frac{2}{M}\Big)^{-1}\Big[\Big(1+\frac{2M}{2(m+1)M+1}\Big)\lambda + \frac{\frac{2}{M}}{2(m+1)M+1}  \Big] |u| \\
        &\le \Big(1-\frac{2}{M}\Big)^{-1}\Big[\Big(1+\frac{1}{M}\Big)\lambda + \frac{2}{M} \Big] |u|
        < \frac{1}{33}|u|.
    \end{align*}
    Here, we used $m \ge M$ and \eqref{eq:M}. In the same way, we also have $|w|<\frac{1}{33}|v|$. Hence, $\tR$ satisfies $C'(\frac{1}{33})$-condition. Since $H$ is a finitely generated $C'(\frac{1}{33})$-group, $H$ is bi-exact by Theorem \ref{thm:main}. Since $G$ is a subgroup of a bi-exact group $H$, $G$ is bi-exact by \cite[Lemma 3.21]{Oy}.
\end{proof}

\begin{rem}
    Note that any finitely generated $C'(\lambda)$-group with $\lambda < \frac{1}{33}$ trivially satisfies the condition of Theorem \ref{thm:infinite case}.
\end{rem}

\providecommand{\bysame}{\leavevmode\hbox to3em{\hrulefill}\thinspace}
\providecommand{\MR}{\relax\ifhmode\unskip\space\fi MR }
\providecommand{\MRhref}[2]{%
  \href{http://www.ams.org/mathscinet-getitem?mr=#1}{#2}
}
\providecommand{\href}[2]{#2}

%\bibliography{small.bib}
%\bibliographystyle{amsalpha}

\vspace{5mm}

\noindent  Department of Mathematics, Vanderbilt University, Nashville 37240, U.S.A.

\noindent E-mail: \emph{koichi.oyakawa@vanderbilt.edu}

\end{document}